\documentclass[12pt, psamsfonts]{amsart}

\usepackage{verbatim,enumerate}
\usepackage[english]{babel}
\usepackage{pgf,tikz}
\usetikzlibrary{arrows}
\usetikzlibrary{calc}

\theoremstyle{plain}

\newtheorem*{theorem*}{Theorem}
\newtheorem{lemma}{Lemma}
\newtheorem*{ctc}{Conway's Thrackle Conjecture}

\theoremstyle{definition}

\newtheorem{remark}{Remark}

\def\deg{\operatorname{deg}}
\def\Im{\operatorname{Im}}
\newcommand{\R}{\mathbb{R}}
\newcommand{\C}{\mathbb{C}}
\newcommand{\Z}{\mathbb{Z}}
\newcommand{\T}{\mathcal{T}}

\begin{document}

\title{Thrackles containing a standard musquash}

\author{Grace Misereh}
\author{Yuri Nikolayevsky}

\address{Department of Mathematics and Statistics, La Trobe University, Melbourne, Australia 3086.}
\email{G.Misere@latrobe.edu.au}
\email{Y.Nikolayevsky@latrobe.edu.au}



\date{\today}

\begin{abstract}
A thrackle is a drawing of a graph in which each pair of edges meets precisely once. Conway's Thrackle Conjecture asserts that a planar thrackle drawing of a graph cannot have more edges than vertices, which is equivalent to saying that no connected component of the graph contains more than one cycle. We prove that a thrackle drawing containing a standard musquash (standard $n$-gonal thrackle) cannot contain any other cycle of length three or five.
\end{abstract}

\maketitle


\section{Introduction}
\label{section:intro}

Let $G$ be a finite simple graph with $n$ vertices and $m$ edges. A \emph{thrackle drawing} of $G$ on the plane is a drawing $\T:G\rightarrow\R^2$, in which every pair of edges meets precisely once, either at a common vertex or at a point of proper crossing (see \cite{LPS97} for definitions of a drawing of a graph and a proper crossing). The notion of thrackle was introduced in the late sixties by John Conway, in relation with the following conjecture.

\begin{ctc}
For a thrackle drawing of a graph on the plane, one has $m\leq n$.
\end{ctc}

Despite considerable effort, the conjecture remains wide open. At present, there are three main approaches to investigating the Conjecture.
The first one, which was pioneered in \cite{LPS97}, is to relax the definition of the thrackle: instead of requiring that the edges meet exactly once, one requires that every pair of edges meets an odd number of times (either at a proper crossing or at a common vertex). The resulting graph drawing is called a \emph{generalised thrackle}. Generalised thrackles are much more flexible than ``genuine" thrackles and are easier to study (in particular, one can use methods of low-dimensional homology theory as in \cite{GY2000, GMY2004, GY2009}). This approach can produce the upper bounds for the number of edges; the best known one as of today is $m < 1.4n$ obtained by Xu in \cite{YX2014} using a theorem of Archdeacon and Stor. This improves earlier upper bounds of \cite{FP2011, GY2000, LPS97}. However, it seems unlikely that this method alone could lead to the full resolution of the Conjecture, since generalised thrackles are much more flexible than thrackles.

The second approach is to prove the Conjecture within specific classes of drawings: straight line thrackles (\cite[\S4]{E1946}; see also an elegant proof by Perles in \cite{PJS2011}), monotone thrackles \cite{PJS2011}, outerplanar and alternating thrackles \cite{GY2012}, and spherical thrackles \cite{GKY2015}. The philosophy of this approach is the fact that sometimes topological results can be proved using geometry. This leads to the natural question, \emph{what is the best thrackle drawing of a given graph?} The answer to this cannot be a straight-edge drawing, as any even cycle of length at least six can be thrackled, but no such cycle has straight-line thrackle drawing. However, we know no example of a thrackle which cannot be deformed to a spherical thrackle (a thrackle on the sphere whose edges are arcs of great circles). Moving further in this direction, one can show that any thrackle can be drawn on the punctured sphere endowed with the hyperbolic metric, with the vertices at infinity, and with the \emph{geodesic} edges.

The third approach is the study of thrackles with small number of vertices. A folklore fact is that the Thrackle Conjecture is true for graphs having at most $11$ vertices. In \cite{FP2011} it is shown that no bipartite graph of up to $11$ vertices (in particular, no graph containing two non-disjoint six-cycles) can be thrackled. It is further proved that for any $\varepsilon > 0$, the inequality $m < (1+\varepsilon) n$ for a thrackled graph follows from the fact that a finite number of graphs (dumbbells) cannot be thrackled.

Note that a complete classification of graphs that can be drawn as thrackles, \emph{assuming} Conway's Thrackle Conjecture, was given in \cite{WOO71}.

The simplest example of a thrackled cycle is the \emph{standard $n$-musquash}, where $n \ge 3$ is odd: distribute $n$ vertices evenly on a circle and then join by an edge every pair of vertices at the maximal distance from each other. This defines a musquash in the sense of Woodall \cite{WOO71}: \emph{$n$-gonal musquash} is a thrackled $n$-cycle whose successive edges $e_0,\dots,e_{n-1}$ intersect in the following manner: if the edge $e_0$ intersects the edges $e_{k_1},\dots,e_{k_{n-3}}$ in that order, then for all $j=1,\dots,n-1$, the edge $e_j$ intersects the edges $e_{k_1+j},\dots,e_{k_{n-3}+j}$ in that order, where the edge subscripts are computed modulo $n$. A complete classification of musquashes was obtained in  \cite{GD1999,GD2001}: every musquash is either isotopic to a standard $n$-musquash, or is a thrackled six-cycle.

\smallskip

Conway's Thrackle Conjecture is equivalent to the fact that no connected component of a thrackled graph $G$ may contain more than one cycle (and to the fact that no figure-eight graph can be thrackled). We prove the following.

\begin{theorem*} 
Let $\T(G)$ be a thrackle drawing of a graph $G$ such that the drawing of a cycle $c \subset G$ is a standard musquash. Then $G$ contains no $3$- and no $5$-cycles (except possibly for $c$ itself).
\end{theorem*}

Note that a thrackled graph can never contain a 4-cycle \cite{LPS97}. So the theorem may be rephrased as follows: if there is a counter-example to Conway's thrackle conjecture that is a figure-eight graph comprised of a standard $n$-musquash and an $m$-cycle sharing a common vertex, then $m$ is at least 6.

\medskip

In this paper, we will consider thrackles up to isotopy, which is to be understood as follows. We regard graph drawing as being drawn on the 2-sphere $S^2$. Then  two drawings $\T_1(G),\, \T_2(G)$ of a graph $G$ are \emph{isotopic} if there is a homeomorphism $h$ of $S^2$ with $\T_2(G)=h(T_1(G))$. Hence an isotopy amounts to a continuous deformation in the plane, combined eventually with an inversion. This notion is more convenient than simple planar deformation as it allows statements such as the following: up to isotopy, the only thrackle drawing of the $5$-cycle is the standard $5$-musquash.

\section{Proof of the Theorem} 
\label{section:th}

In this section we give the proof of the Theorem assuming some technical lemmas that we establish later in Section~\ref{section:lemmas}.

Suppose that in a thrackle drawing $\T(G)$ of a graph $G$, a cycle $c$ is thrackled as a standard musquash. Then $n:=l(c)$ is odd, where $l(c)$ denotes the length of $c$. Suppose that the graph $G$ contains another cycle $c'$ with $l(c') = 3$ or $l(c') = 5$. For convenience, we can remove from $G$ all the other edges and vertices; that is, we may assume that $G = c \cup c'$.

\medskip

The first step is to reduce the proof to the case when $G$ is a figure-eight graph consisting of cycles $c$ and $c'$, of the same lengths as before, sharing a common vertex and such that $\T(c)$ is still a standard musquash. As both $c$ and $c'$ are odd, they cannot be disjoint in $G$ \cite[Lemma~2.1(ii)]{LPS97}, and so $c \cap c'$ is a nonempty union of vertices and paths. Repeatedly using the vertex-splitting operation shown in Figure~\ref{figure:6}(a), we obtain a new thrackle drawing such that $c\cap c'$ is a union of vertices, without changing the lengths of the cycles $c, c'$ and without violating the fact that $c$ is thrackled as a standard musquash. Next, we can perturb the drawing in a neighbourhood of every vertex at which $\T(c)$ and $\T(c')$ meet without crossing, to replace a vertex of degree four by four crossings (Figure~\ref{figure:6}(b)).

\begin{figure}[h]
\centering
\begin{minipage}[b]{6cm}
\begin{tikzpicture}[scale=0.9,>=triangle 45]
\draw[thick] (-0.7,3.42)-- (-0.7,0.56);
\draw[thick] (-0.7,3.42)-- (-1.32,2.9);
\draw[thick] (-0.7,3.42)-- (-0.14,2.98); 
\draw[thick] (-1.2,2.4)-- (-0.2,2.4);
\draw[thick] (-1.2,1.95)-- (-0.2,1.95);
\draw[thick] (-1.2,1.5)-- (-0.2,1.5); 
\draw[->, very thick] (0.73,2.1)-- (2.43,2.1); 
\draw[thick] (3.54,3.54)-- (4,0.56);
\draw[thick] (4,0.56)-- (4.36,3.58);
\draw[thick] (3.54,3.54)-- (4.76,2.72);
\draw[thick] (4.36,3.58)-- (3.32,2.7); 
\draw[thick] (3.36,2.4)-- (4.72,2.4);
\draw[thick] (3.36,1.95)-- (4.72,1.95);
\draw[thick] (3.36,1.5)-- (4.72,1.5); 
\draw(1.53,0) node {(a)};
\fill  (-0.7,3.42) circle (1.8pt);
\fill  (-0.7,0.56) circle (1.8pt);
\fill  (3.54,3.54) circle (1.8pt);
\fill  (4,0.56) circle (1.8pt);
\fill  (4.36,3.58) circle (1.8pt);
\end{tikzpicture}
\end{minipage}
\hspace{0.3cm}
\begin{minipage}[b]{6cm}
\begin{tikzpicture}[scale=0.55,>=triangle 45]
\draw[thick] (0,0)--(2,2)--(0,4);
\draw[thick] (2,2) to [out=90,in=-180] (5,4);
\draw[thick] (5,0) to [out=180,in=-90] (2,2);
\fill  (2,2) circle (3.0pt);
\draw[->, very thick] (5.36,1.9)-- (8.14,1.9); 
\draw[thick] (8.5,0)--(10.5,2)--(8.5,4);
\fill  (10.5,2) circle (3pt);
\coordinate (A) at (10,3); \fill (A) circle (3pt);
\draw[thick] (13.5,0) to [out=180,in=-130] (A);
\draw[thick] (A) to [out=-90,in=150] (10.5,1.5);
\draw[thick] (10.5,1.5) to [out=-10,in=-180] (13.5,4);
\draw(0.25,0.8) node {$c$};
\draw(4.65,0.7) node {$c'$};
\draw(8.75,0.8) node {$c$};
\draw(13.15,0.7) node {$c'$};
\draw(6.5,-1.5) node {(b)};
\end{tikzpicture}
\end{minipage}
\caption{Vertex splitting (a) and replacing a vertex of degree four by four crossings (b).} 
\label{figure:6}
\end{figure}
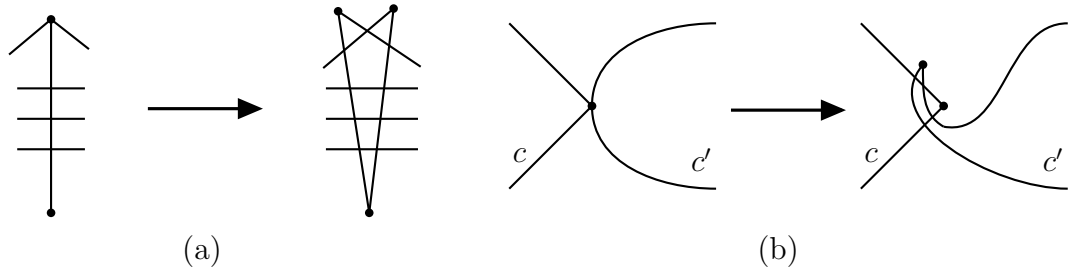

Now $c\cap c'$ is a set of vertices at each of which the drawings of the cycles $c, c'$ cross. Counting the number of crossings of the closed curves $\T(c)$ and $\T(c')$ we find that the number of such vertices must be odd \cite{LPS97}. If there are at least three of them, then since $l(c') \leq 5$, we arrive at a contradiction with \cite[Theorem~3]{GY2012}. So $c \cap c'$ is a single vertex, hence $G=c \cup c'$ is a figure-eight graph, with common vertex $v$ say.

\medskip

In the second step, we fix a positive orientation on $\T(c)$, and to every domain $D$ of the complement $\R^2\setminus \T(c)$, assign a non-negative integer label, the rotation number of the closed oriented curve $\T(c)$ around a point of $D$. The labels change from $0$ for the outer domain to $\frac{1}{2}(n-1)$ for the innermost $n$-gonal domain (Figure~\ref{figure:1}).

\begin{figure}[h]
\begin{tikzpicture}[scale=0.8]
\centering
\foreach \x in {0,1,...,6} {
\draw ({1.1*cos((2*pi*\x/7+pi/2+pi/7) r)},{1.1*sin((2*pi*\x/7+pi/2+pi/7) r)}) node {$2$};
\draw ({2*cos((2*pi*\x/7+pi/2) r)},{2*sin((2*pi*\x/7+pi/2) r)}) node {$1$};
\fill  ({4*cos((2*pi*\x/7+pi/2) r)},{4*sin((2*pi*\x/7+pi/2) r)}) circle (2.5pt);
\draw[thick] ({4*cos((2*pi*\x/7+pi/2) r)},{4*sin((2*pi*\x/7+pi/2) r)}) -- ({4*cos((2*pi*(\x+3)/7+pi/2) r)},{4*sin((2*pi*(\x+3)/7+pi/2) r)});
}
\draw (0,0) node {$3$};
\draw (2.5,3.2) node {$0$};
\end{tikzpicture}
\caption{Labels (rotation numbers) for a standard $7$-musquash.}
\label{figure:1}
\end{figure}
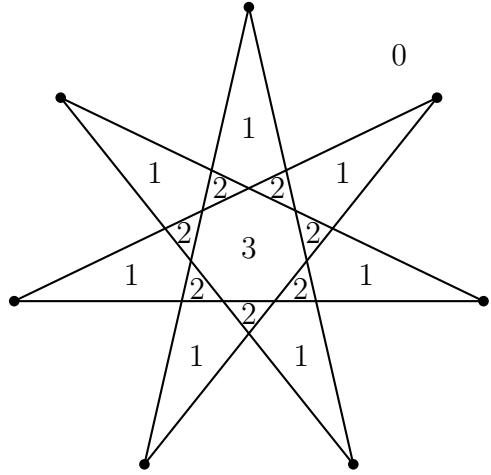

We need the following two lemmas in which we consider the mutual position of a standard musquash and an edge or a two-path attached to it.
\begin{lemma}
\label{lemma:1edge}
The endpoints of a curve crossing all the edges of a standard musquash exactly once and not passing through the vertices (so that the resulting drawing is a thrackle) lie in domains whose labels differ by one.
\end{lemma}

Note that for the purposes of the proof of the Theorem, we will not need Lemma~\ref{lemma:1edge} in its full generality; we will only require special cases. The proof of Lemma~\ref{lemma:1edge} is quite involved, and the proof of the conclusions we require could be simpler. On the other hand, Lemma~\ref{lemma:1edge} may be of a certain independent interest and may be useful for the further study of thrackles containing musquashes. We also observe that Lemma~\ref{lemma:1edge} generalises Theorem~3 of \cite{GY2012}.

Now consider a graph $G'$ consisting of an odd cycle $c$ and a two-path $p = vuw$ attached to a vertex $v \in c$. Suppose that $G'$ is thrackled in such a way that $\T(c)$ is a standard musquash and that the starting segment of $\T(vu)$ lies in the outer domain of $\T(c)$. A possible thrackle drawing is shown in Figure~\ref{figure:3}. Slightly perturbing the drawing of the edge $vu$ in a neighbourhood of $v$ we obtain a thrackle drawing of the disjoint union of $c$ and $p$, with $\T(c)$ still being the standard musquash, and the starting point of $\T(p)$ lying in the outer domain (so that $v$ is no longer a vertex of $c$). Then by Lemma~\ref{lemma:1edge}, the point $u$ lies in a domain labelled $1$, and then $w$, the other endpoint of $p$, lies in a domain with the label either $0$ or $2$. The following lemma, which will be crucial for the proof of the Theorem for $l(c')=5$, states that the second case cannot occur.

\begin{lemma}
\label{lemma:2path}
Let a graph $G'$ consist of an odd cycle $c$ and a $2$-path $p = vuw$ attached to a vertex $v \in c$. Let $\T(G')$ be a thrackle drawing, with $\T(c)$ a standard musquash, such that the starting segment of $\T(vu)$ lies in the outer domain of $\T(c)$. Then $u$ lies in the domain labelled $1$ and $w$ lies in the outer domain of $\T(c)$.
\end{lemma}

\begin{figure}[h]
\begin{tikzpicture}[scale=0.75]
\centering
\draw[thick] (7.01,2.18)-- (5.78,-4.22);
\draw[thick] (5.78,-4.22)-- (9.67,1.01);
\draw[thick] (9.67,1.01)-- (3.89,-2.02);
\draw[thick] (3.89,-2.02)-- (10.41,-1.79);
\draw[thick] (10.41,-1.79)-- (4.44,0.83);
\draw[thick] (4.44,0.83)-- (8.68,-4.12);
\draw[thick] (8.68,-4.12)-- (7.01,2.18);
\fill  (5.78,-4.22) circle (2.5pt);
\draw[thick](5.94,-3.96) node {};
\fill  (8.68,-4.12) circle (2.5pt);
\draw[thick](8.84,-3.86) node {};
\fill  (10.41,-1.79) circle (2.5pt);
\draw[thick](10.56,-1.54) node {};
\fill  (9.67,1.01) circle (2.5pt);
\draw[thick](9.82,1.28) node {};
\fill  (7.01,2.18) circle (2.5pt);
\draw (7.18,2.44) node {$v$};
\fill  (4.44,0.83) circle (2.5pt);
\fill  (3.89,-2.02) circle (2.5pt);
\draw[thick] (4.06,-1.76) node {};
\draw[thick] (7.01,2.18) to [out=-45,in=-90] (10,0.5);
\draw[thick] (10,0.5) to[out= 90,in=10] (7.18,3);
\draw[thick] (7.18,3) to[out=195,in=-150] (4.5,-3);
\draw[thick] (4.5,-3) to[out=30,in=-90] (5,-1.76);
\fill  (5,-1.76) circle (2.5pt);
\draw (5.2,-1.65) node {$u$};
\draw[thick] (5,-1.76) to[out=-70,in=-180] (6,-5);
\draw[thick] (6,-5) to[out=0,in=-90] (11,-1.54);
\draw[thick] (11,-1.54) to[out=90,in=-20] (4.5,-1);
\fill  (4.5,-1) circle (2.5pt);
\draw (4.67,-0.74) node {$w$};
\end{tikzpicture}
\caption{A standard $7$-musquash with a $2$-path attached.}
\label{figure:3}
\end{figure}
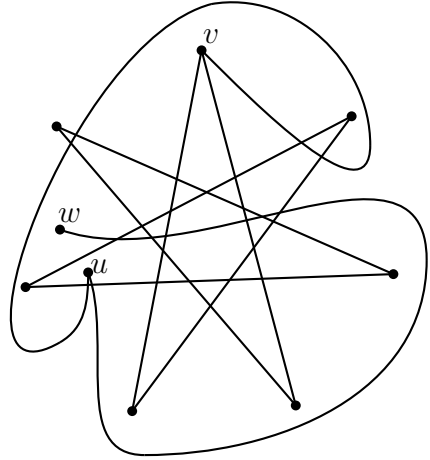

\begin{remark}\label{rem:paths}
Combining Lemma~\ref{lemma:1edge} and Lemma~\ref{lemma:2path} one can generalise Lemma~\ref{lemma:2path} to the case when $p$ is a three-path: if $p=vv_1v_2v_3$, then $v_1$ and $v_3$ lie in the domain labelled $1$, and $v_2$, in the outer domain. It would be very interesting to know, if the direct generalisation of this fact for longer paths $p$ is still true: is it so that a path attached to a vertex of a standard musquash and starting at the outer domain cannot get ``too deep" in the musquash (most optimistically, if it always ends in a domain labelled either $0$ or $1$)?
\end{remark}

Returning to the proof of the Theorem, we have a figure-eight graph $G$ consisting of an odd cycle $c$ and an odd cycle $c'$ of length $3$ or $5$ that share a common vertex $v$, and a thrackle drawing $\T(G)$ such that $\T(c)$ is a standard musquash. From the above argument, or by \cite[Lemma~2.2]{LPS97}, the drawings of $c$ and $c'$  cross at their common vertex $v$. Of the two edges of $c'$ incident to $v$, one goes to the inner domain of the musquash $\T(c)$ (its starting segment lies in the domain labelled $1$) and then ends in the outer domain. The other edge goes to the outer domain. Hence by Lemma~\ref{lemma:1edge} and Lemma~\ref{lemma:2path}, because $c'$ has length $3$ or $5$, we obtain the following key fact: \emph{all the vertices of $c'$ other than $v$ lie in  domains labelled $0$ or $1$}. 

\medskip

The third step in the proof of the Theorem is the operation of \emph{edge removal} \cite[Section~4]{GY2012}, which will enable us to eventually shorten $c$ to a thrackled cycle of the same length as $c'$. 

The operation of edge removal is inverse to Woodall's edge insertion operation \cite[Figure~14]{WOO71}. Let $\T(H)$ be a thrackle drawing of a graph $H$ and let $v_1v_2v_3v_4$ be a three-path in $H$ such that $\deg v_2 = \deg v_3 = 2$. Let $A = \T (v_1v_2) \cap \T (v_3v_4)$. Removing the drawing of the edge $v_2v_3$, together with the segments $Av_2$ and $Av_3$ of $\T (v_1v_2)$ and $\T (v_3v_4)$ from the point $A$ to $\T(v_2)$ and $\T(v_3)$, respectively, we obtain a drawing of a graph with a single edge $v_1v_4$ in place of the three-path $v_1v_2v_3v_4$ (Figure~\ref{figure:2}). (In what follows, to make the notation less cumbersome, we will use the vertex name $v_i$ to denote the point $\T(v_i)$, when there is no risk of confusion).

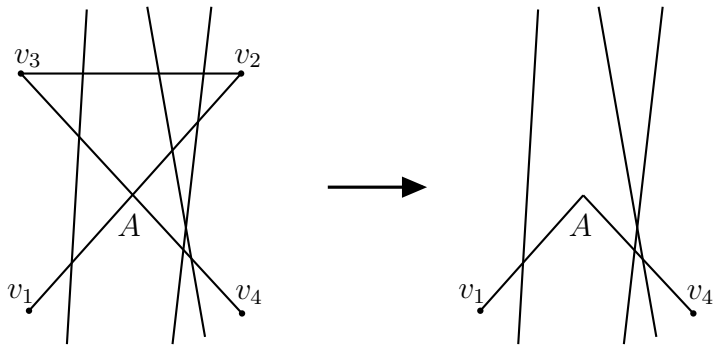
\begin{figure}[h]
\begin{tikzpicture}[scale=0.6,>=triangle 45]
\draw[thick] (-1,3.66)-- (3.88,3.66); 
\draw[thick] (3.88,3.66)-- (-0.82,-1.6); 
\draw[thick] (-1,3.66)-- (3.9,-1.66); 
\draw[thick] (0.46,5.08)-- (0.02,-2.34); 
\draw[thick] (1.8,5.14)-- (3.08,-2.18); 
\draw[thick] (3.22,5.14)-- (2.36,-2.34); 
\draw[->, very thick] (5.8,1.14) -- (8,1.14);
\draw[thick] (9.18,-1.6) -- (11.46,0.96); 
\draw[thick] (11.46,0.96) -- (13.9,-1.66); 
\draw[thick] (10.46,5.08)-- (10.02,-2.34); 
\draw[thick] (11.8,5.14)-- (13.08,-2.18); 
\draw[thick] (13.22,5.14)-- (12.36,-2.34); 
\fill  (-1,3.66) circle (2.0pt); \draw (-0.84,4) node {$v_3$}; 
\fill  (3.88,3.66) circle (2.0pt); \draw (4.04,4) node {$v_2$}; 
\fill  (-0.82,-1.6) circle (2.0pt); \draw (-1,-1.22) node {$v_1$}; 
\fill  (3.9,-1.66) circle (2.0pt); \draw (4.06,-1.24) node {$v_4$}; 
\draw (1.4,0.3) node {$A$};  
\fill  (9.18,-1.6) circle (2.0pt); \draw (9,-1.22) node {$v_1$}; 
\fill  (13.9,-1.66) circle (2.0pt); \draw (14.06,-1.24) node {$v_4$}; 
\draw(11.4,0.3) node {$A$};
\end{tikzpicture}
\caption{The edge removal operation.}
\label{figure:2}
\end{figure}

Unlike edge insertion, edge removal does not necessarily result in a thrackle drawing. Consider the triangular domain $\bigtriangleup$ bounded by the arcs $\T(v_2v_3)$, $Av_2$ and $v_3A$. We have the following Lemma.
\begin{lemma}[{\cite[Lemma~3]{GY2012}}]
\label{lemma:triangle}
Edge removal results in a thrackle drawing if and only if $\bigtriangleup$ contains no vertices of $\T(G)$.
\end{lemma}

It follows that edge removal is always possible on (every edge of) a standard musquash. The resulting thrackled cycle is outerplanar (all the vertices lie on the boundary of a single domain) and as such, is Reidemeister equivalent to a standard musquash by \cite[Theorem~1]{GY2012}. In fact, by the following lemma, it is even \emph{isotopic} to a  standard musquash (which will be important for the argument that follows) -- see Figure~\ref{figure:5}.

\begin{lemma}
\label{lemma:mtom}
The edge removal operation on a standard musquash of length $n\geq 5$ results in a standard $(n-2)$-musquash.
\end{lemma}

\begin{figure}[h]
\begin{tikzpicture}[scale=0.8,>=triangle 45]
\centering
\foreach \x in {0,1,...,8} {
\fill  ({4*cos((2*pi*\x/9+pi/2) r)},{4*sin((2*pi*\x/9+pi/2) r)}) circle (2.5pt);
\draw[thick] ({4*cos((2*pi*\x/9+pi/2) r)},{4*sin((2*pi*\x/9+pi/2) r)}) -- ({4*cos((2*pi*(\x+4)/9+pi/2) r)},{4*sin((2*pi*(\x+4)/9+pi/2) r)});
}
\draw[->, very thick] (4.75,0)-- (6.25,0); 
\foreach \x in {-1,0,1,3,4,5,6}
\fill  ({10+4*cos((2*pi*\x/9+pi/2) r)},{4*sin((2*pi*\x/9+pi/2) r)}) circle (2.5pt);
\foreach \x in {-1,0,1,4,5,6}
\draw[thick] ({10+4*cos((2*pi*\x/9+pi/2) r)},{4*sin((2*pi*\x/9+pi/2) r)}) -- ({10+4*cos((2*pi*(\x+4)/9+pi/2) r)},{4*sin((2*pi*(\x+4)/9+pi/2) r)});
\coordinate (A) at ({4*cos((2*pi*2/9+pi/2) r)},{4*sin((2*pi*2/9+pi/2) r)});
\coordinate (B) at ({4*cos((-2*pi*2/9+pi/2) r)},{4*sin((-2*pi*2/9+pi/2) r)});
\coordinate (C) at (0,{4*sin((8*pi/9) r)/(cos((4*pi/9+pi/2) r)-cos((12*pi/9+pi/2) r))});
\draw [very thick] (A)--(B); \draw [very thick] (A)--(C); \draw [very thick] (C)--(B);
\coordinate (D) at ({10+4*cos((2*pi*3/9+pi/2) r)},{4*sin((2*pi*3/9+pi/2) r)});
\coordinate (E) at ({10+4*cos((-2*pi*3/9+pi/2) r)},{4*sin((-2*pi*3/9+pi/2) r)});
\coordinate (F) at (10,{4*sin((8*pi/9) r)/(cos((4*pi/9+pi/2) r)-cos((12*pi/9+pi/2) r))});
\draw [thick] (D)--(F); \draw [thick] (F)--(E);
\draw ({4*cos((2*pi*2/9+pi/2) r)+1.5},{4*sin((2*pi*2/9+pi/2) r)+0.3}) node {$e$};
\end{tikzpicture}
\caption{The edge removal operation on the edge $e$ of the standard $9$-musquash results in a standard $7$-musquash.}
\label{figure:5}
\end{figure}
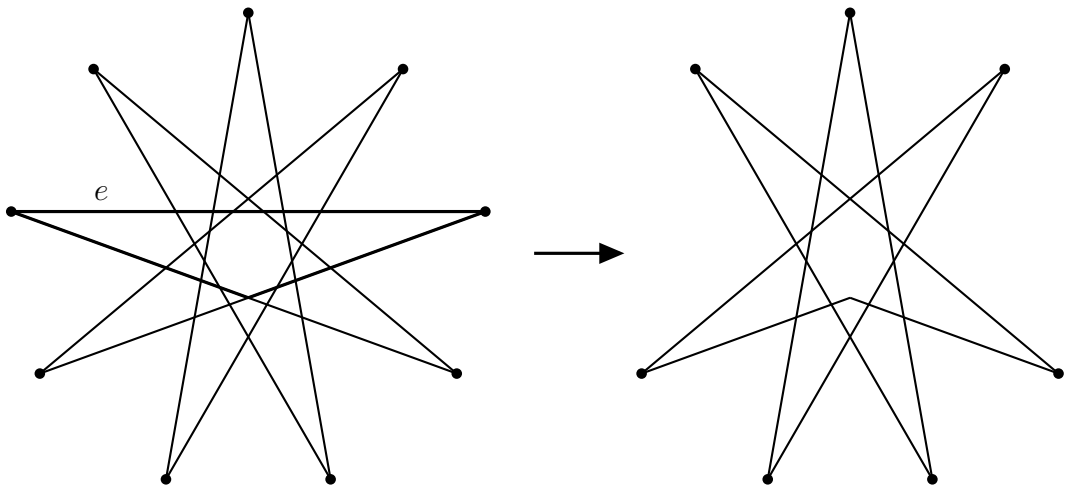

What is more, for our thrackle drawing $\T(G)$, edge removal is always possible on $\T(c)$ as long as $c$ is longer than $c'$. To see that, we first consider the case when $l(c')=3$. We cannot remove the two edges of $c$ incident to the common vertex $v$ (as $\deg v=4$) and also the edges whose respective triangular domains contain the vertices of $c'$, by Lemma~\ref{lemma:triangle}. But the two vertices of $c'$ other than $v$ lie in the domains labelled $0$ and $1$, one in each. Clearly, the vertex lying in the outer domain does not belong to any triangular domain, while the vertex lying in the domain labelled $1$ belongs to exactly two triangular domains, which prohibits the removal of two edges. Hence if $l(c) \ge 5$, there is always an edge of $c$ which can be removed. By Lemma~\ref{lemma:triangle}, after edge removal, the resulting drawing is again a thrackle drawing of the figure-eight graph consisting of an odd cycle $c^*$ of length $l(c)-2$ and the cycle $c'$, and by Lemma~\ref{lemma:mtom}, the drawing of $c^*$ is again a standard musquash. Repeatedly using this argument we obtain a thrackle drawing of a figure-eight graph consisting of two three-cycles, which is a contradiction (as can be seen by inspection or by \cite[Lemma~5(b)]{GY2000}). This completes the proof in the case $l(c')=3$.
The proof in the case $l(c') = 5$ is almost identical, up to the second last step. This time, from among the five vertices of $c'$, other than $v$, two lie in the outer domain of $\T(c)$, and the other two, in the domains labelled $1$. Together with the two edges of $c$ incident to $v$ this gives at most six edges of $c$ spared from removal. Therefore, by repeatedly edge-removing we get a thrackled figure-eight graph consisting of two five-cycles.

The proof is then completed by the following Lemma\footnote{for which we supply a direct proof, although the Thrackle Conjecture is most likely known to be true for graphs of nine vertices -- see the Introduction.}.
\begin{lemma}
\label{lemma:no55}
A figure-eight graph consisting of two five-cycles has no thrackle drawing. 
\end{lemma}

\section{Proof of the Lemmas}
\label{section:lemmas}

\begin{proof}[Proof of Lemma~\ref{lemma:1edge}]
Let $c$ be an odd cycle of length $n$ and let $\T(c)$ be its standard musquash drawing. We assume the edges of $\T(c)$ to be straight line segments and the vertices to be the vertices of a regular $n$-gon inscribed in the unit circle $C$ bounding the closed unit disc $D$. Let $\gamma=AB$ be a simple curve crossing every edge of $\T(c)$ exactly once and not passing through the vertices. We can assume that both endpoints of $\gamma$ lie inside $C$, and that $\gamma$ meets $C$ in a finite collection of proper crossings $A_1, A_2, \dots, A_k$ labelled in the direction from $A$ to $B$, where $k \ge 0$ is even. Then $\gamma =AA_1 \cup \bigcup_{j=1}^{k-1}A_jA_{j+1}\cup A_kB$, the arcs $AA_1, A_2A_3, \dots, A_kB$ lie in $D$, and the arcs $A_1A_2, A_3A_4, \dots, A_{k-1}A_k$, outside $C$. For each arc in this decomposition, consider the set of edges of $\T(c)$ it crosses. The arcs lying outside $C$ do not meet $\T(c)$ at all. An arc lying in $D$ crosses an edge of $\T(c)$ if and only if its endpoints lie on the opposite sides of that edge. Therefore the set of edges of $\T(c)$ such an arc is crossing depends only on its endpoints, and we lose no generality by replacing that arc by a straight line segment with the same endpoints. Note also that $\gamma$ cannot completely lie inside $C$ (so that $k \ge 1$), as no straight line crosses all the edges of $\T(c)$ (since one of the half-planes, determined by such a straight line, must contain more than the half of the vertices of $\T(c)$ and hence contain two  vertices adjacent in $c$).

An arc $A_jA_{j+1}$ lying in $D$, with both endpoints on $C$, crosses an \emph{even number of consecutive edges} of $\T(c)$: the points $A_j, A_{j+1}$ split $C$ into two segments, hence splitting the set of vertices of $\T(c)$ into two subsets; the arc $A_jA_{j+1}$ crosses all the edges incident to the vertices of the smaller of these two subsets. As a subset of $c$, the union of edges crossed by $A_j A_{j+1}$ is an even path.

The picture is more complicated for the arcs $AA_1, A_kB$ having one endpoint in the interior of $D$. Let $XY$ be an arc lying in $D$, with exactly one endpoint $Y$ on $C$, and let  $S\subset c $ be the union of edges of $\T(c)$ which $XY$ crosses. Then both $S$ and its complement $c\setminus S$ is a finite collection of paths. Denote $O_1(XY)$ the number of paths of odd length in $S$, and $O_0(XY)$ the number of paths of odd length in $c\setminus S$. We have the following key lemma.
\begin{lemma}
\label{lemma:o1o0}
If $X$ lies in a domain labelled $s$, then $O_1(X Y) = s$ and $O_0(XY)= s \pm 1$.
\end{lemma}
Assuming Lemma~\ref{lemma:o1o0}, we can complete the proof of Lemma~\ref{lemma:1edge} as follows. Let $S_1, S_2$ be the unions of edges of $c$ crossed by the arcs $AA_1, BA_k$, respectively. Since $\gamma$ crosses every edge of $\T(c)$ exactly once, the (interiors of the) sets $S_1, S_2$ are disjoint, and the union of $S_1 \cup S_2$ and the sets of edges of $\T(c)$ which are crossed by the arcs $A_jA_{j+1}$ is the whole cycle $c$. So, since the union of edges which are crossed by an arc $A_jA_{j+1}$ is an even path in $c$ (which can be empty), every connected component of the complement $c\setminus(S_1 \cup S_2)$ must be a path of even length. For this to be true, $S_1$ has to contain at least as many odd paths as $c\setminus S_2$ does, and vice versa, so $O_1(AA_1) \geq O_0(BA_k)$ and	 $O_1(BA_k) \geq O_0(AA_1)$. 
Thus, if the points $A, B$ lie in domains labelled $s_1,s_2$, respectively, then by Lemma~\ref{lemma:o1o0} we get $s_1  \geq s_2 - 1$ and $s_2  \geq s_1- 1$, so $s_1 -1 \leq s_2 \leq s_1 + 1$. As $s_1 \neq s_2$, since $\gamma$ has an odd number of crossings with the closed curve $\T(c)$, we obtain $s_2 = s_1 \pm 1$, as claimed.
\end{proof}


\begin{proof}[Proof of Lemma~\ref{lemma:o1o0}]
Our argument above shows that we can assume $XY$ to be a straight line segment, so the proof of the lemma reduces to a question in plane geometry: we have to find the edges of $\T(c)$ which are crossed by $XY$. We can further assume that $s \ne 0$. If $s=0$ we can take $X$ also lying on $C$ and the above arguments show that the union of edges of $\T(c)$ crossed by $XY$ is an even path of $c$ (which can be empty), so $O_1(X Y) = 0$ and $O_0(XY)= 1$.

Let $n=2m+1$. We place the vertices of $\T(c)$ at the points $e^{\pi i/n}, e^{3\pi i/n}, \dots e^{(2n-1)\pi i/n} \in \C=\R^2$ and label the edges of $\T(c)$ in such a way that the $j^{\mathrm{th}}$ edge joins the vertices $e^{(2mj+1)\pi i/n}$ and $e^{(2m(j+1)+1)\pi i/n}$ where $j = 0,1,\dots, n-1$. From the symmetry and by a slight perturbation, we can assume that the point $X$ lies in the open angle $\arg \ z \in (0, \frac{\pi }{n})$, as in Figure~\ref{figure:4}.

\begin{figure}[h]
\begin{tikzpicture}[scale=0.9]
\foreach \x in {0,1,3,4,5,7,8}
\fill  ({((2*pi*\x/9+pi/9) r)}:4) circle (2.5pt);
\coordinate (A3) at ({((5*pi/9) r)}:4);
\coordinate (A2) at ({((3*pi/9) r)}:4);
\coordinate (A7) at ({((-5*pi/9) r)}:4);
\coordinate (A8) at ({((-3*pi/9) r)}:4);
\coordinate (C1) at ($0.95*(A7)+0.05*(A2)$); \coordinate (C3) at ($0.9*(A7)+0.1*(A2)$); \draw[thick] (A2)--(C3); \draw[thick,dashed] (C3)--(C1);
\coordinate (C2) at ($0.95*(A3)+0.05*(A8)$); \coordinate (C4) at ($0.9*(A3)+0.1*(A8)$); \draw[thick] (A8)--(C4); \draw[thick,dashed] (C2)--(C4);
\foreach \x in {0,-1,3,1,4,5} \draw[thick] ({((2*pi*\x/9+pi/9) r)}:4) -- ({((2*pi*(\x+4)/9+pi/9) r)}:4);
\draw (0,0) circle (4);
\coordinate (X) at ({((pi/24) r)}:0.95); \fill  (X) circle (2.5pt); \draw ($ (X) + (0.27,0.24) $) node {$X$};
\coordinate (Y) at ({((pi/24) r)}:4); \fill  (Y) circle (2.5pt); \draw ($ (Y) + (0.27,0.24) $) node {$Y$};
\draw ($ ({((pi/9) r)}:4) + (-0.5,0.15) $) node {$0$};
\draw ($ ({((pi/9) r)}:4) + (-0.5,-0.7) $) node {$n\!-\!1$};
\draw ($ ({((3*pi/9) r)}:4) + (-1,-0.4) $) node {$n\!-\!2$};
\draw ($ ({((3*pi/9) r)}:4) + (0.25,-0.9) $) node {$n\!-\!3$};
\draw ($ ({((7*pi/9) r)}:4) + (0.1,-0.5) $) node {$3$};
\draw ($ ({((7*pi/9) r)}:4) + (0.5,0) $) node {$2$};
\draw (-3.7,0.3) node {$0$};
\draw (-3.7,-0.3) node {$1$};
\draw ($ ({((-pi/9) r)}:4) + (-0.4,0.55) $) node {$2$};
\draw ($ ({((-pi/9) r)}:4) + (-0.45,-0.2) $) node {$1$};
\draw ($ ({((-3*pi/9) r)}:4) + (0,0.7) $) node {$4$};
\draw ($ ({((-3*pi/9) r)}:4) + (-0.55,0.25) $) node {$3$};
\draw ($ ({((11*pi/9) r)}:4) + (0.15,0.95) $) node {$n\!-\!2$};
\draw ($ ({((11*pi/9) r)}:4) + (1.05,0.2) $) node {$n\!-\!1$};
\end{tikzpicture}
\caption{Labelling the edges of a musquash.}
\label{figure:4}
\end{figure}
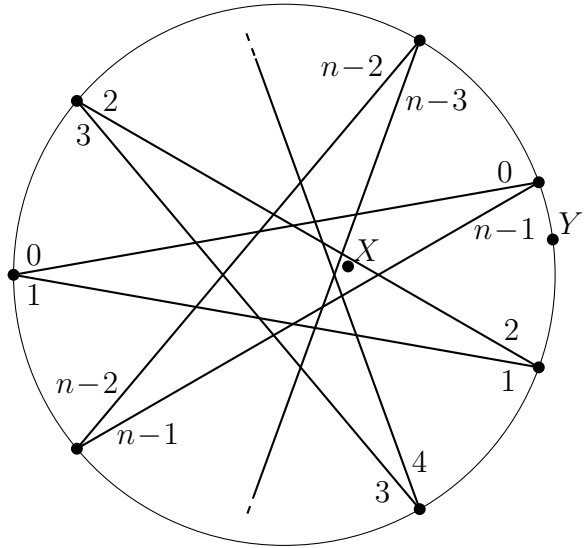

First consider the case when the point $Y$ lies on the radius of $C$ passing through $X$. Then $X=re^{i \alpha}, \; Y=e^{i \alpha}$, where $0 < r < 1, \; 0 < \alpha < \frac{\pi }{n}$. The segment $XY$ crosses the edge $j$ of $\T(c)$ if and only if it crosses the line containing that edge, whose equation is
\begin{equation*}
\Im(z(e^{-(2m(j+1)+1)\pi i/n}- e^{-(2mj+1)\pi i/n})) + \sin(\pi / n) = 0.
\end{equation*}
Hence the segment $XY$ crosses the edge $j$ when $-\Im(e^{i \alpha-(2m(j+1)+1)\pi i/n}-e^{i \alpha-(2mj+1)\pi i/n}) \in (\sin (\frac{\pi}{n}), \frac{1}{r}\sin(\frac{\pi}{n}))$ which is equivalent to
\[
\sin \frac{\pi}{n}< (-1)^j\left(\sin \Big(\alpha +\frac{j\pi}{n}\Big)+\sin \Big(\alpha +\frac{(j-1)\pi}{n}\Big)\right)<\frac{1}{r}\sin \frac{\pi}{n},
\]
which is equivalent again to the following condition:
\begin{equation} \label{equation:1}
\sin \frac{\pi}{2n}< (-1)^j\sin \Big(\alpha +\frac{(2j-1)\pi}{2n}\Big)<\frac{1}{r}\sin \frac{\pi}{2n} \, .
\end{equation}
If $j = 0$, the left-hand inequality of \eqref{equation:1} is false since $\alpha-\frac{\pi}{2n}\in (-\frac{\pi}{2n},\frac{\pi}{2n})$. For all the other values of $j$, we have $\sin(\alpha +\frac{(2j-1)\pi}{2n})> 0$, and so, to satisfy \eqref{equation:1}, $j$ must be nonzero even. Denote $s(j)=\sin(\alpha + \frac{(2j-1)\pi}{2n})$. From the fact that $0 < \alpha < \frac{\pi}{n}$, it follows that
\[
\sin \frac{\pi}{2n}< s(2m)<s(2)<s(2m-2)<s(4)<\dots<s(M)<1,\]
where
\[
M:=\begin{cases}
m/2&:m\ \text{is even}\\
(m+1)/2&:\text{otherwise}.
\end{cases}
\]
As all the crossings of the radial segment $XY$ with the edges of $\T(c)$ have the same orientation, the fact that $X$ lies in the domain labelled $s$ implies that there must be exactly $s$ crossings. Hence the set of the values of $j$ satisfying \eqref{equation:1} is the set of the first $s$ terms of the sequence $(2m, 2, 2m - 2, 4, \dots,M)$, that is, $\{2,4, \dots, 2[\frac{s}2], 2m-2[\frac{s-1}2], 2m-2, 2m\}$. 

We now consider the general case when $Y$ does not necessarily lie on the radius of $C$ passing through $X$. For $Y \in C$, let $V(Y)$ be the $n$-dimensional row-vector over $\Z_2$ whose components are labelled from $0$ to $n - 1$, such that $V(Y)_j = 1$, if $XY$ crosses the edge $j$ of $\T(c)$, and $V(Y)_j = 0$ otherwise. As we have just shown, if $XY$ lies on a radius of $C$, then
\begin{equation}
\label{equation:V}
V(Y) = (0,(0,1)^{[s/2]}, 0^{n-2s-1}, (0,1)^{[(s+1)/2]}),
\end{equation}
where the superscript denotes the number of consecutive repeats of the sequence.

The number $O_1(XY)$ (respectively, $O_0(XY)$) is the number of odd sequences of consecutive ones (respectively, zeros) in the vector $V(Y)$ (counted in the cyclic order, so that if $V(Y)_{n-1-a}= \dots = V(Y)_{n-1}=V(Y)_0= \dots = V(Y)_b$ for some $a, b \ge 0$, we count it as a single sequence of length $a+b+2$). For the vector $V(Y)$ in \eqref{equation:V} we have $O_1(XY) = s$ and $O_0(XY) = s - 1$.

When $Y \in C$ moves in the positive direction, the vector $V(Y)$ only changes when $Y$ passes through the vertices of $\T(c)$. Initially $Y$ lies between the vertices $e^{(2n-1)\pi i/n}$ and $e^{\pi i/n}$, so when $Y$ passes through the vertex $e^{\pi i/n}$, the resulting vector $V(Y)$ is obtained from the one in \eqref{equation:V} by adding the vector $(1, 0,\dots, 0, 1)$, which gives $V(Y) = ((1,0)^{[(s+2)/2]}, 0^{n-2s-1}, (1,0)^{[(s-1)/2]},0)$, so $O_1(XY) = s$ and $O_0(XY) = s - 1$.

When $Y$ keeps moving, every time when it passes through the vertex $e^{(2j+1)\pi i/n}, \; j = 1,\dots, n-1$, we add to $V(Y)$ the vector $(0^{n-2j-1},1^2,0^{2j-1})$ for $j=1, \dots, \frac{n-1}2$, and the vector $(0^{2n-2j-1},1^2,0^{2j-n-1})$ for $j=\frac{n+1}2, \dots, n-1$, as the number of crossings of $XY$ with the two edges incident to that vertex changes from $0$ to $1$ or vice versa. Then a routine check shows that $O_1(XY)$ and $O_0(XY)$ take the following values, as claimed.
\renewcommand{\arraystretch}{1.2}
\begin{table}[h]
\centering
\begin{center}
\begin{tabular}{|c|c|c|c|}
  \hline
  Interval for $j$ & $V(Y)$ & $O_1$ & $O_0$ \\
  \hline
  $\big[ 1, [\frac{s-1}2] \big]$ & $((1,0)^{[\frac{s+2}2]}, 0^{n-2s-1}, (1,0)^{[\frac{s-1}2]-j}, (0,1)^{j},0)$ & $s$ & $s - 1$ \\
  $\big[ [\frac{s-1}2]+1, \frac{n-3}2-[\frac{s}2]\big]$ & $((1,0)^{[\frac{s+2}2]}, 0^{n-2j-3-2[\frac{s}2]}, 1^{2(j-[\frac{s-1}2])}, (0,1)^{[\frac{s-1}2]},0)$ & $s$ & $s + 1$ \\
  $\big[ \frac{n-3}2-[\frac{s}2]+1, \frac{n-3}2 \big]$ & $((1,0)^{\frac{n-1}2-j}, (0,1)^{j-\frac{n-3}2+[\frac{s}2]}, 1^{n-2s-1}, (0,1)^{[\frac{s-1}2]},0)$ & $s$ & $s - 1$ \\
  $\frac{n-1}2$ & $(0,(1,0)^{[\frac{s}2]}, 1^{n-2s-1}, (1,0)^{[\frac{s+1}2]})$ & $s$ & $s - 1$ \\
  $\big[ \frac{n+1}2, \frac{n-1}2+[\frac{s+1}2]\big]$ & $(0,(1,0)^{[\frac{s}2]}, 1^{n-2s-1}, (1,0)^{[\frac{s+1}2]+\frac{n-1}{2}-j}, (0,1)^{j-\frac{n-1}{2}})$ & $s$ & $s - 1$ \\
  $\!\! \big[ \frac{n+1}2+[\frac{s+1}2], n-2-[\frac{s}2] \big]$ \!\!\!\!& $(0,(1,0)^{[\frac{s}2]}, 1^{2(n-[\frac{s}2]-1-j)}, 0^{2j-n+1-2[\frac{s+1}2]}, (0,1)^{[\frac{s+1}2]})$ & $s$ & $s + 1$ \\
  $\big[ n-1-[\frac{s}2], n-1 \big]$ & $(0,(1,0)^{n-1-j}, (0,1)^{[\frac{s}2]-n+1+j}, 0^{n-2s-1}, (0,1)^{[\frac{s+1}2]})$ & $s$ & $s - 1$ \\
  \hline
\end{tabular}
\end{center}
\caption{\ }
\label{table:VO}
\qedhere
\end{table}
\end{proof}

\begin{proof}[Proof of Lemma~\ref{lemma:mtom}]
Let $c$ be an odd cycle of length $n$. Choose a direction on $c$ and label the edges $0, 1, \dots , n - 1$ in consecutive order. According to \cite{GD2001}, the standard musquash is uniquely determined by its \emph{unsigned crossing table}, that is, by the order of crossings on every edge with the other edges. For the standard musquash $\T(c)$, this order on the edge labelled $i$ is
\begin{equation}\label{eq:crossing}
i + n - 3, i + n - 5, \dots , i + 4, i + 2, i + n - 2, i + n - 4, \dots ,i + 5, i + 3,
\end{equation}
where the labels are computed modulo $n$ \cite{GD2001}. By Lemma~\ref{lemma:triangle}, the edge removal operation on any edge results in a thrackle drawing $\T(c^*)$ of a cycle of length $n-2$. Without loss of generality we assume that we remove the edge labelled $n-2$. We keep the labels $0, 1, \dots, n-4$ for the edges of $c^*$ which are unaffected by the removal, and we label $n-3$ the single edge of $c^*$ formed by the segments of edges $n-3$ and $n-1$ of $c$ as the result of the edge removal.

The proof now is just a routine verification that the order of crossings for every edge of $\T(c^*)$ is the same as that given by \eqref{eq:crossing}, with $n$ replaced by $n-2$, and with the labels computed modulo $n-2$. We consider three cases.

Suppose an edge $i$ of $\T(c)$ crosses all the three edges $n - 3, n- 2, n - 1$ (that is, $1 \le i \le n-5$). If $i$ is even, then from \eqref{eq:crossing} the order of crossings is $i - 3, i - 5, \dots, 1, n-1$, $n-3, \dots, i + 4, i + 2, i - 2, i - 4, \dots , 2, 0, n-2, n-4, \dots, i + 5, i + 3$, so the crossings with $n-3$ and $n-1$ are consecutive. Hence the crossing order on the edge $i$ of $\T(c^*)$ is obtained by deleting the labels $n - 1$ and $n - 2$, which results in the same sequence as in \eqref{eq:crossing}, with $n$ replaced by $n-2$. If $i$ is odd, then the order of crossings is $i - 3, i - 5,$ $\dots, 2, 0, n-2, \dots, i + 4, i + 2, i - 2, i - 4, \dots, 1, n-1, n-3, \dots, i + 5, i + 3$, and the proof follows by a similar argument.

Suppose now an edge $i$ of $\T(c)$ crosses only two of the three edges $n-3, n-2, n-1$, so that $i=0$ or $i=n-4$. When $i=0$, by \eqref{eq:crossing} the crossing order is $n - 3, n - 5, \dots , 4, 2, n - 2,$ $n - 4, \dots , 5, 3$. The crossing order on the edge $0$ of $\T(c^*)$ is obtained by deleting the labels $n - 3$ and $n - 2$, which results in the same sequence, with $n$ replaced by $n-2$. Similarly, for $i=n-4$, \eqref{eq:crossing} gives $n - 7, n - 9, \dots , 0, n - 2, n - 6, n - 8, \dots , 1, n - 1$. The crossing order on the edge $n-4$ of $\T(c^*)$ is obtained by deleting the labels $n - 1$ and $n - 2$. The resulting sequence is the same as that obtained from \eqref{eq:crossing} by replacing $n$ by $n-2$,  and then reducing modulo $n-2$. 

And finally, the crossing order on the edge $n-3$ of $\T(c^*)$ is the crossing order on the edge $n-3$ of $\T(c)$, up to but excluding the crossing with the edge $n-1$, followed by the crossing order on the edge $n-1$  of $\T(c)$ starting from but excluding the crossing with the edge $n-3$. From \eqref{eq:crossing} we obtain the sequence $n - 6, n - 8, \dots , 1, n - 5, \dots , 4, 2$. This sequence is the same as that obtained from \eqref{eq:crossing} by replacing $n$ by $n-2$,  and then reducing modulo $n-2$. 
\end{proof}


\begin{proof}[Proof of Lemma~\ref{lemma:2path}]
The fact that $u$ lies in a domain labelled $1$ follows from Lemma~\ref{lemma:1edge}. Then, again by Lemma~\ref{lemma:1edge}, $w$ lies either in the outer domain, or in a domain labelled $2$. Arguing by contradiction, suppose that $w$ lies in a domain labelled $2$.
Our approach is to shorten the musquash $\T(c)$ to a standard musquash of length at most $7$ using the edge removal operation. In view of Lemma~\ref{lemma:triangle}, the edge removal operation on $c$ is forbidden on the following edges: on the two edges incident to $v$ as $\deg v > 2$, on the two edges whose corresponding triangular domains $\triangle$ contain the vertex $u$, and on the four edges whose corresponding triangular domains $\triangle$ contain the vertex $w$ (it is not hard to see that if a point lies in a domain with label $i < \frac{n-1}{2}$, then it is contained in $2i$ triangular domains, hence forbidding the removal of $2i$ edges; a point lying in the domain with label $\frac{n-1}{2}$, the innermost $n$-gon of the musquash, lies in all the triangular domains, hence not permitting any edge removal at all). This gives no more than $8$ edges in total. So edge removal on $c$ is always possible as long as the cycle $c$ has at least $9$ edges. By Lemma~\ref{lemma:mtom}, edge removal results in a standard musquash, and what is more, the vertices $u$ and $w$ still lie in the domains labelled $1$ and $2$ respectively, of the complement to that new musquash. Repeating edge removal, we come to a thrackle consisting of a standard $7$-musquash, with a two-path attached to its vertex. Note that for some pairs of domains containing $u$ and $w$, it could happen that the sets of forbidden edges overlap, which could make further edge removal possible, hence resulting in a standard $5$-musquash, with a two-path attached to its vertex.
Lemma~\ref{lemma:2path} is obvious when $c$ is a 3-cycle. So Lemma~\ref{lemma:2path} follows form the following result.\end{proof}

\begin{lemma}\label{lemma:2pathon5or7}
Lemma~\ref{lemma:2path} holds when $c$ is a 5-cycle or a 7-cycle.
\end{lemma}

\begin{remark} In order to complete the proof of the Theorem, it remains to prove Lemmas~\ref{lemma:2pathon5or7} and~\ref{lemma:no55}, which deal with thrackles that are sufficiently small that they can be treated by computer (which is why we separated Lemma~\ref{lemma:2pathon5or7} from Lemma~\ref{lemma:2path}). 
Nevertheless, we will now give formal proofs, in part so that the paper does not rely on computer proof, and in part to display just how combinatorially complicated such simple thrackle problems can be. It would be interesting if simpler proofs of these two lemmas could be found.
\end{remark}

\begin{proof}[Proof of Lemma~\ref{lemma:2pathon5or7}]
Before beginning in earnest, let us introduce some terminology, and make some observations, valid where $c$ has any odd order $n\geq 3$. We employ the approach used in the proof of Lemma~\ref{lemma:1edge}. Inscribe $\T(c)$ in a circle $C$ bounding the disc $D$, and consider a simple curve $\gamma=AB$ which crosses all the edges of $\T(c)$ exactly once without passing through the vertices. Without loss of generality we can decompose $\gamma$ into the union of segments, $\gamma =AA_1 \cup \bigcup_{j=1}^{k-1}A_jA_{j+1}\cup A_kB$, where the arcs $AA_1, A_2A_3, \dots, A_kB$ lie in $D$, and the arcs $A_1A_2, A_3A_4, \dots, A_{k-1}A_k$, outside $C$. Furthermore, we can assume that $AA_1$ and $A_kB$ are straight line segments. Now consider an arc $A_jA_{j+1}$  in $D$. The set of edges which it crosses can be found as follows: the points $A_j, A_{j+1}$ split the set of vertices of $\T(c)$ into two subsets lying on $C$. Then the arc $A_jA_{j+1}$ crosses all the edges incident to the vertices of the smaller of that two subsets. The idea now is to push $A_jA_{j+1}$ as much as possible off $D$, to a small neighbourhood of $C$. Thus, by a sequence of Reidemeister moves, we replace $A_jA_{j+1}$ by a union of small straight line segments (with the endpoints on $C$), each lying in a neighbourhood of a vertex and crossing the two edges incident to it, and a union of arcs lying outside of $D$ joining (in the original order) the endpoints of these small straight line segments. Let $\gamma'$ be the curve obtained by performing this  replacement on each arc $A_jA_{j+1}$ in $D$. The curve $\gamma'$ is simple, has the same endpoints $A, B$ and crosses all the edges of $\T(c)$ exactly once. We call $\gamma' \cap D$ the \emph{diagram} of $\gamma$.
\begin{figure}[h]
\begin{tikzpicture}[scale=0.75]
\centering
\foreach \x in {0,1,...,8}{
\fill  ({((2*pi*\x/9+pi/9) r)}:4) circle (2.5pt);
\fill  ($({((2*pi*\x/9+pi/9) r)}:4)+(10,0)$) circle (2.5pt);
\draw[thick] ($({((2*pi*\x/9+pi/9) r)}:4)+(10,0)$) -- ($({((2*pi*(\x+4)/9+pi/9) r)}:4)+(10,0)$);
\draw[thick] ({((2*pi*\x/9+pi/9) r)}:4) -- ({((2*pi*(\x+4)/9+pi/9) r)}:4);
}
\draw (10,0) circle (4);
\coordinate (X) at ({((pi/24) r)}:0.95); \fill  (X) circle (2.5pt);
\coordinate (Z) at ($({((pi/24) r)}:0.95)+(10,0)$); \fill  (Z) circle (2.5pt);
\coordinate (Y) at ($({((pi/24) r)}:4)+(10,0)$); \fill  (Y) circle (2.5pt);
\draw[thick] (Z)--(Y);
\coordinate (A) at ($({((3*pi/9+pi/12) r)}:4)+(10,0)$); \fill (A) circle (2.5pt);
\coordinate (B) at ($({((3*pi/9-pi/12) r)}:4)+(10,0)$); \fill (B) circle (2.5pt);
\draw[thick] (A)--(B);
\coordinate (C) at ($({((5*pi/9+pi/12) r)}:4)+(10,0)$); \fill (C) circle (2.5pt);
\coordinate (D) at ($({((5*pi/9-pi/12) r)}:4)+(10,0)$); \fill (D) circle (2.5pt);
\draw[thick] (C)--(D);
\coordinate (E) at ($({((pi+pi/12) r)}:4)+(10,0)$); \fill (E) circle (2.5pt);
\coordinate (F) at ($({((pi-pi/12) r)}:4)+(10,0)$); \fill (F) circle (2.5pt);
\draw[thick] (E)--(F);
\coordinate (X2) at ({((7*pi/9+pi/24) r)}:1.8); \fill  (X2) circle (2.5pt);
\coordinate (Z2) at ($({((7*pi/9+pi/24) r)}:1.8)+(10,0)$); \fill  (Z2) circle (2.5pt);
\coordinate (Y2) at ($({((7*pi/9+pi/24) r)}:4)+(10,0)$); \fill  (Y2) circle (2.5pt);
\draw[thick] (Z2)--(Y2);
\coordinate (G) at ({((pi/9+pi/54) r)}:4.25);
\draw[thick] (X) to [out= -50,in=15] (G);
\coordinate (H) at ({((7*pi/9-pi/21) r)}:4.1);
\draw[thick] (G) to [out= -165,in=-63] (H);
\coordinate (J) at ({((pi+pi/24) r)}:4.1);
\draw[thick] (H) to [out= 117,in=120] (J);
\draw[thick] (J) to [out= -60,in=-120] (X2);
\end{tikzpicture}
\caption{A curve and its diagram.}
\label{figure:diagram}
\end{figure}
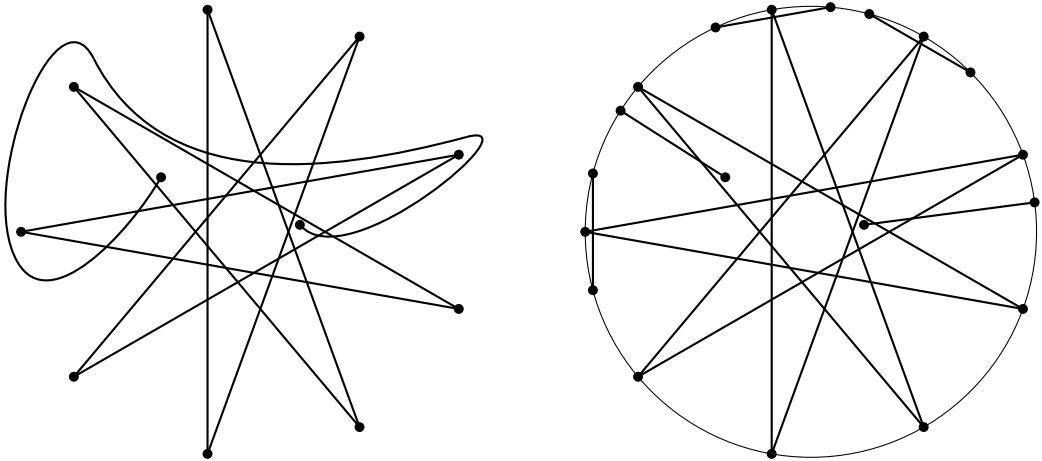

Note that different (Reidemeister inequivalent) curves $\gamma$ may have the same diagram, as one can join the endpoints of the straight line segment lying on $C$ by arcs lying outside $D$ in different ways.

Now suppose that $c$ is a 5-cycle, drawn as $5$-musquash, with a two-path $vuw$ attached to vertex $v$. 
Consider the diagram for the edge $wu$. There is only one domain labelled $2$ for the $5$-musquash and, up to rotation, only one possible starting segment  $wA^1, \; A^1 \in C$; we assume $\arg A^1\in (0,\pi/10)$ (here and below we use the superscripts rather than subscripts for the points on $C$, as a priori we do not know in which order they are connected by the arcs lying outside $D$). Place the musquash and label its edges as in the proof of Lemma~\ref{lemma:o1o0} (so that the vertices are at the $5^{\text{th}}$ roots of $-1$) -- see Figure~\ref{figure:diagrams5}. Again as in the proof of Lemma~\ref{lemma:o1o0}, let $V(w)$ be the row-vector in $\Z_2^5$ whose components are labelled $0, \dots, 4$ such that $V(w)_j=1$ if $wA^1$ crosses the edge $j$, and $V(w)_j=0$ otherwise. By \eqref{equation:V} $V(w)=(0,0,1,0,1)$. Now consider the segment $uA^2, \; A^2 \in C$, of the diagram of $uw$. At first sight, the point $u$ could be placed in one of the five domains labelled $1$, and for each position of $u$ there are five possible choices of $A^2$. From Table~\ref{table:VO} with $n=5, \; s=1$ we find that $V(u)$ is one of the following five vectors, \emph{up to a cyclic permutation} (where initially we choose $u$ in the angle $\arg u \in (0, \pi/5)$):
\begin{equation}\label{eq:V5u}
    (0,0,0,0,1), \quad (1,0,0,0,0), \quad (1,0,1,1,0), \quad (0,1,1,1,0), \quad (0,1,1,0,1).
\end{equation}
Now, the sets of edges crossed by $wA^1$ and by $uA^2$ must be disjoint, and the complement to their union must be the union of even paths in $c$ (or empty). Up to reflection, this gives only two possibilities for $V(u)$: either the first vector in \eqref{eq:V5u} cyclicly shifted by one position to the left, or the third vector in \eqref{eq:V5u} cyclicly shifted by two positions to the left. The corresponding diagrams are given in Figure~\ref{figure:diagrams5} (where we also added the outside arcs, as they are determined uniquely, up to isotopy).

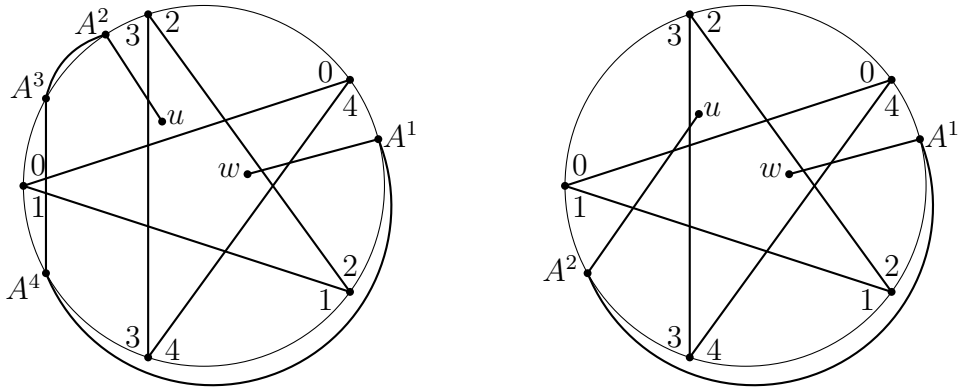
\begin{figure}[h]
\begin{tikzpicture}[scale=0.6]
\foreach \x in {0,1,...,4} {
\fill  ({4*cos((2*pi*\x/5+pi/5) r)},{4*sin((2*pi*\x/5+pi/5) r)}) circle (2.5pt);
\draw ({3.7*cos((2*pi*(2*\x)/5+pi/5+pi/25) r)},{3.7*sin((2*pi*(2*\x)/5+pi/5+pi/25) r)}) node {$\x$};
\draw ({3.7*cos((2*pi*(2*\x+2)/5+pi/5-pi/25) r)},{3.7*sin((2*pi*(2*\x+2)/5+pi/5-pi/25) r)}) node {$\x$};
\draw[thick] ({4*cos((2*pi*\x/5+pi/5) r)},{4*sin((2*pi*\x/5+pi/5) r)}) -- ({4*cos((2*pi*(\x+2)/5+pi/5) r)},{4*sin((2*pi*(\x+2)/5+pi/5) r)});
\fill  ({12+4*cos((2*pi*\x/5+pi/5) r)},{4*sin((2*pi*\x/5+pi/5) r)}) circle (2.5pt);
\draw ({12+3.7*cos((2*pi*(2*\x)/5+pi/5+pi/25) r)},{3.7*sin((2*pi*(2*\x)/5+pi/5+pi/25) r)}) node {$\x$};
\draw ({12+3.7*cos((2*pi*(2*\x+2)/5+pi/5-pi/25) r)},{3.7*sin((2*pi*(2*\x+2)/5+pi/5-pi/25) r)}) node {$\x$};
\draw[thick] ({12+4*cos((2*pi*\x/5+pi/5) r)},{4*sin((2*pi*\x/5+pi/5) r)}) -- ({12+4*cos((2*pi*(\x+2)/5+pi/5) r)},{4*sin((2*pi*(\x+2)/5+pi/5) r)});
}
%
%
\draw (0,0) circle (4);
\coordinate (w) at  ({cos((pi/12) r)},{sin((pi/12) r)}); \fill  (w) circle (2.5pt); \draw ($ (w) + (-0.4,0.1) $) node {$w$};
\coordinate (A1) at ({4*cos((pi/12) r)},{4*sin((pi/12) r)}); \fill  (A1) circle (2.5pt); \draw ($ (A1) + (0.5,0.1) $) node {$A^1$};
\draw [thick] (w)--(A1);
\coordinate (u) at  ({1.7*cos((3*pi/5+pi/12) r)},{1.7*sin((3*pi/5+pi/12) r)}); \fill  (u) circle (2.5pt); \draw ($ (u) + (0.3,0.1) $) node {$u$};
\coordinate (A2) at ({4*cos((3*pi/5+pi/12) r)},{4*sin((3*pi/5+pi/12) r)}); \fill  (A2) circle (2.5pt); \draw ($ (A2) + (-0.3,0.3) $) node {$A^2$};
\draw [thick] (u)--(A2);
\coordinate (A3) at (-3.5,{sqrt(16-3.5^2)}); \fill  (A3) circle (2.5pt); \draw ($ (A3) + (-0.4,0.3) $) node {$A^3$};
\coordinate (A4) at (-3.5,{-sqrt(16-3.5^2)}); \fill  (A4) circle (2.5pt); \draw ($ (A4) + (-0.5,-0.3) $) node {$A^4$};
\draw [thick] (A3)--(A4);
\draw[thick] (A2) to [bend right] (A3);
\draw[thick] let \p1 = ($ 0.5*(A4) - 0.5*(A1) $), \n1 = {veclen(\x1,\y1)}, \n3={atan(\y1/\x1)} in (A1) arc (\n3:\n3-180:\n1);
%
%
\draw (12,0) circle (4);
\coordinate (w) at  ({12+cos((pi/12) r)},{sin((pi/12) r)}); \fill  (w) circle (2.5pt); \draw ($ (w) + (-0.4,0.1) $) node {$w$};
\coordinate (A1) at ({12+4*cos((pi/12) r)},{4*sin((pi/12) r)}); \fill  (A1) circle (2.5pt); \draw ($ (A1) + (0.5,0.1) $) node {$A^1$};
\draw [thick] (w)--(A1);
\coordinate (u) at  ({12+1.9*cos((3*pi/5+pi/12) r)},{1.9*sin((3*pi/5+pi/12) r)}); \fill  (u) circle (2.5pt); \draw ($ (u) + (0.3,0.1) $) node {$u$};
\coordinate (A2) at ({12-3.5},{-sqrt(16-3.5^2)}); \fill  (A2) circle (2.5pt); \draw ($ (A2) + (-0.6,0.1) $) node {$A^2$};
\draw [thick] (u)--(A2);
\draw[thick] let \p1 = ($ 0.5*(A2) - 0.5*(A1) $), \n1 = {veclen(\x1,\y1)}, \n3={atan(\y1/\x1)} in (A1) arc (\n3:\n3-180:\n1);
\end{tikzpicture}
\caption{Diagrams of $wu$ for the $5$-musquash.}
\label{figure:diagrams5}
\end{figure}

We next consider the edge $uv$. By slightly perturbing the drawing in a neighbourhood of $v$ we can move $v$ into the outer domain. Then, by construction, in the diagram of the edge $vu$, the first two edges of the musquash $\T(c)$ that $uv$ crosses, counting from $v$, are incident to the same vertex of $\T(c)$. What is more, as $v$ lies in the outer domain, the starting segment $uB^1, \; B^1 \in C$, of the diagram of $uv$ has to be chosen in such a way that the complement to the set of edges of $\T(c)$ which it crosses is a union of even paths. From \eqref{eq:V5u} we see that this is only possible for the first two vectors, so that  $uB^1$ only crosses a single edge. The choice of that edge determines the diagram of $uv$ uniquely by the fact that the remaining straight lines segment in it are short segments close to the vertices of $\T(c)$ crossing pairs of edges incident to those vertices. We can now add the diagram of $uv$ to the diagram of $wu$ on the left in Figure~\ref{figure:diagrams5} in the three possible ways shown in Figure \ref{figure:5uvw1}, where in the diagram on the right in Figure~\ref{figure:5uvw1}, there are two possible ways of adding the segment $B^4B^5$. Now, in the diagram on the left in Figure~\ref{figure:5uvw1}, the point $B^1$ can only be connected to $B^4$, and then the only way to make a thrackle pass is to join $B^4$ to $B^3$ \emph{inside} the disc $D$, going around the point $w$. But then the last two crossings on the edge $uv$ counting from $u$ are not with two consecutive edges of $\T(c)$, and so the edge $uv$ cannot be deformed to one attached to a vertex of $\T(c)$. In the diagram in the middle in Figure~\ref{figure:5uvw1}, the point $B^1$ cannot be joined to any other point $B^j$. In the diagram on the right in Figure~\ref{figure:5uvw1}, the segment $B^2B^3$ is unreachable from the point $B^1$, no matter which of two possibilities for segment $B^4B^5$ are chosen.

\begin{figure}[h!]
\begin{tikzpicture}[scale=0.6]
\foreach \x in {0,1,...,4} {
\fill  ({4*cos((2*pi*\x/5+pi/5) r)},{4*sin((2*pi*\x/5+pi/5) r)}) circle (2.5pt);
\draw[thick] ({4*cos((2*pi*\x/5+pi/5) r)},{4*sin((2*pi*\x/5+pi/5) r)}) -- ({4*cos((2*pi*(\x+2)/5+pi/5) r)},{4*sin((2*pi*(\x+2)/5+pi/5) r)});
}
\draw (0,0) circle (4);
\coordinate (w) at  ({cos((pi/12) r)},{sin((pi/12) r)});
\fill  (w) circle (2.5pt);
\draw ($ (w) + (-0.4,0.1) $) node {$w$};
\coordinate (A1) at ({4*cos((pi/12) r)},{4*sin((pi/12) r)});
\fill  (A1) circle (2.5pt);
\draw ($ (A1) + (0.5,0.1) $) node {$A^1$};
\draw [thick] (w)--(A1);
\coordinate (u) at  ({1.7*cos((3*pi/5+pi/12) r)},{1.7*sin((3*pi/5+pi/12) r)});
\fill  (u) circle (2.5pt);
\draw ($ (u) + (0.3,0.1) $) node {$u$};
\coordinate (A2) at ({4*cos((3*pi/5+pi/12) r)},{4*sin((3*pi/5+pi/12) r)});
\fill  (A2) circle (2.5pt);
\draw ($ (A2) + (-0.3,0.3) $) node {$A^2$};
\draw [thick] (u)--(A2);
\coordinate (A3) at (-3.5,{sqrt(16-3.5^2)});
\fill  (A3) circle (2.5pt);
\draw ($ (A3) + (-0.4,0.3) $) node {$A^3$};
\coordinate (A4) at (-3.5,{-sqrt(16-3.5^2)});
\fill  (A4) circle (2.5pt);
\draw ($ (A4) + (-0.5,-0.3) $) node {$A^4$};
\draw [thick] (A3)--(A4);
\draw[thick] (A2) to [bend right] (A3); 
\draw[thick] let \p1 = ($ 0.5*(A4) - 0.5*(A1) $),
\n1 = {veclen(\x1,\y1)}, \n3={atan(\y1/\x1)}
in (A1) arc (\n3:\n3-180:\n1);
\coordinate (B1) at ({4*cos((3*pi/5+pi/24) r)},{4*sin((3*pi/5+pi/24) r)});
\fill  (B1) circle (2.5pt);
\draw ($ (B1) + (0.1,0.5) $) node {$B^1$};
\draw [thick] (u)--(B1);
\coordinate (B2) at ({4*cos((-pi/5+pi/10) r)},{4*sin((-pi/5+pi/10) r)});
\fill  (B2) circle (2.5pt);
\draw ($ (B2) + (-0.4,0.3) $) node {$B^2$};
\coordinate (B3) at ({4*cos((-pi/5-pi/10) r)},{4*sin((-pi/5-pi/10) r)});
\fill  (B3) circle (2.5pt);
\draw ($ (B3) + (-0.4,0.2) $) node {$B^3$};
\draw [thick] (B3)--(B2);
\coordinate (B4) at ({4*cos((pi/5+pi/10) r)},{4*sin((pi/5+pi/10) r)});
\fill  (B4) circle (2.5pt);
\draw ($ (B4) + (0.4,0.3) $) node {$B^4$};
\coordinate (B5) at ({4*cos((pi/5-pi/12) r)},{4*sin((pi/5-pi/12) r)});
\fill  (B5) circle (2.5pt);
\draw ($ (B5) + (0.4,0.4) $) node {$B^5$};
\draw [thick] (B4)--(B5);
%
\foreach \x in {0,1,...,4} {
\fill  ({9+4*cos((2*pi*\x/5+pi/5) r)},{4*sin((2*pi*\x/5+pi/5) r)}) circle (2.5pt);
\draw[thick] ({9+4*cos((2*pi*\x/5+pi/5) r)},{4*sin((2*pi*\x/5+pi/5) r)}) -- ({9+4*cos((2*pi*(\x+2)/5+pi/5) r)},{4*sin((2*pi*(\x+2)/5+pi/5) r)});
}
\draw (9,0) circle (4);
\coordinate (w) at  ({9+cos((pi/12) r)},{sin((pi/12) r)});
\fill  (w) circle (2.5pt);
\draw ($ (w) + (-0.4,0.1) $) node {$w$};
\coordinate (A1) at ({9+4*cos((pi/12) r)},{4*sin((pi/12) r)});
\fill  (A1) circle (2.5pt);
\draw ($ (A1) + (0.5,0.1) $) node {$A^1$};
\draw [thick] (w)--(A1);
\coordinate (u) at  ({9+1.7*cos((3*pi/5+pi/12) r)},{1.7*sin((3*pi/5+pi/12) r)});
\fill  (u) circle (2.5pt);
\draw ($ (u) + (0.3,0.1) $) node {$u$};
\coordinate (A2) at ({9+4*cos((3*pi/5+pi/12) r)},{4*sin((3*pi/5+pi/12) r)});
\fill  (A2) circle (2.5pt);
\draw ($ (A2) + (-0.3,0.3) $) node {$A^2$};
\draw [thick] (u)--(A2);
\coordinate (A3) at (9-3.5,{sqrt(16-3.5^2)});
\fill  (A3) circle (2.5pt);
\draw ($ (A3) + (-0.4,0.3) $) node {$A^3$};
\coordinate (A4) at (9-3.5,{-sqrt(16-3.5^2)});
\fill  (A4) circle (2.5pt);
\draw ($ (A4) + (-0.5,-0.3) $) node {$A^4$};
\draw [thick] (A3)--(A4);
\draw[thick] (A2) to [bend right] (A3); 
\draw[thick] let \p1 = ($ 0.5*(A4) - 0.5*(A1) $),
\n1 = {veclen(\x1,\y1)}, \n3={atan(\y1/\x1)}
in (A1) arc (\n3:\n3-180:\n1);
\coordinate (B1) at ({9+4*cos((3*pi/5+4*pi/24) r)},{4*sin((3*pi/5+4*pi/24) r)});
\fill  (B1) circle (2.5pt);
\draw ($ (B1) + (0.1,-0.6) $) node {$B^1$};
\draw [thick] (u)--(B1);
\coordinate (B2) at ({9+4*cos((-pi/5+pi/10) r)},{4*sin((-pi/5+pi/10) r)});
\fill  (B2) circle (2.5pt);
\draw ($ (B2) + (-0.4,0.3) $) node {$B^2$};
\coordinate (B3) at ({9+4*cos((-pi/5-pi/10) r)},{4*sin((-pi/5-pi/10) r)});
\fill  (B3) circle (2.5pt);
\draw ($ (B3) + (-0.4,0.2) $) node {$B^3$};
\draw [thick] (B3)--(B2);
\coordinate (B4) at ({9+4*cos((pi/5+pi/10) r)},{4*sin((pi/5+pi/10) r)});
\fill  (B4) circle (2.5pt);
\draw ($ (B4) + (0.4,0.3) $) node {$B^4$};
\coordinate (B5) at ({9+4*cos((pi/5-pi/12) r)},{4*sin((pi/5-pi/12) r)});
\fill  (B5) circle (2.5pt);
\draw ($ (B5) + (0.4,0.4) $) node {$B^5$};
\draw [thick] (B4)--(B5);
%
\foreach \x in {0,1,...,4} {
\fill  ({18+4*cos((2*pi*\x/5+pi/5) r)},{4*sin((2*pi*\x/5+pi/5) r)}) circle (2.5pt);
\draw[thick] ({18+4*cos((2*pi*\x/5+pi/5) r)},{4*sin((2*pi*\x/5+pi/5) r)}) -- ({18+4*cos((2*pi*(\x+2)/5+pi/5) r)},{4*sin((2*pi*(\x+2)/5+pi/5) r)});
}
\draw (18,0) circle (4);
\coordinate (w) at  ({18+cos((pi/12) r)},{sin((pi/12) r)});
\fill  (w) circle (2.5pt);
\draw ($ (w) + (-0.4,0.1) $) node {$w$};
\coordinate (A1) at ({18+4*cos((pi/12) r)},{4*sin((pi/12) r)});
\fill  (A1) circle (2.5pt);
\draw ($ (A1) + (-0.5,-0.5) $) node {$A^1$};
\draw [thick] (w)--(A1);
\coordinate (u) at  ({18+1.7*cos((3*pi/5+pi/12) r)},{1.7*sin((3*pi/5+pi/12) r)});
\fill  (u) circle (2.5pt);
\draw ($ (u) + (0.3,0.1) $) node {$u$};
\coordinate (A2) at ({18+4*cos((3*pi/5+pi/12) r)},{4*sin((3*pi/5+pi/12) r)});
\fill  (A2) circle (2.5pt);
\draw ($ (A2) + (-0.3,0.3) $) node {$A^2$};
\draw [thick] (u)--(A2);
\coordinate (A3) at (18-3.5,{sqrt(16-3.5^2)});
\fill  (A3) circle (2.5pt);
\draw ($ (A3) + (-0.4,0.3) $) node {$A^3$};
\coordinate (A4) at (18-3.5,{-sqrt(16-3.5^2)});
\fill  (A4) circle (2.5pt);
\draw ($ (A4) + (-0.5,-0.3) $) node {$A^4$};
\draw [thick] (A3)--(A4);
\draw[thick] (A2) to [bend right] (A3); 
\draw[thick] let \p1 = ($ 0.5*(A4) - 0.5*(A1) $),
\n1 = {veclen(\x1,\y1)}, \n3={atan(\y1/\x1)}
in (A1) arc (\n3:\n3-180:\n1);
\coordinate (B1) at ({18+4*cos((3*pi/5-pi/12) r)},{4*sin((3*pi/5-pi/12) r)});
\fill  (B1) circle (2.5pt);
\draw ($ (B1) + (0.1,0.5) $) node {$B^1$};
\draw [thick] (u)--(B1);
\coordinate (B2) at ({18+4*cos((-3*pi/5+pi/10) r)},{4*sin((-3*pi/5+pi/10) r)});
\fill  (B2) circle (2.5pt);
\draw ($ (B2) + (0.2,0.5) $) node {$B^2$};
\coordinate (B3) at ({18+4*cos((-3*pi/5-pi/10) r)},{4*sin((-3*pi/5-pi/10) r)});
\fill  (B3) circle (2.5pt);
\draw ($ (B3) + (0.4,0.5) $) node {$B^3$};
\draw [thick] (B3)--(B2);
\coordinate (B5) at (18-3.2,{sqrt(16-3.2^2)});
\fill  (B5) circle (2.5pt);
\draw ($ (B5) + (0.6,-0.2) $) node {$B^5$};
\coordinate (B4) at (18-3.2,{-sqrt(16-3.2^2)});
\fill  (B4) circle (2.5pt);
\draw ($ (B4) + (0.6,0.3) $) node {$B^4$};
\draw [thick,dashed] (B4)--(B5);
\coordinate (B5') at (18-3.8,{sqrt(16-3.8^2)});
\fill  (B5') circle (2.5pt);
\coordinate (B4') at (18-3.8,{-sqrt(16-3.8^2)});
\fill  (B4') circle (2.5pt);
\draw [thick,dashed] (B5')--(B4');
\end{tikzpicture}
\caption{Adding the diagram of $uv$. Case 1.}
\label{figure:5uvw1}
\end{figure}
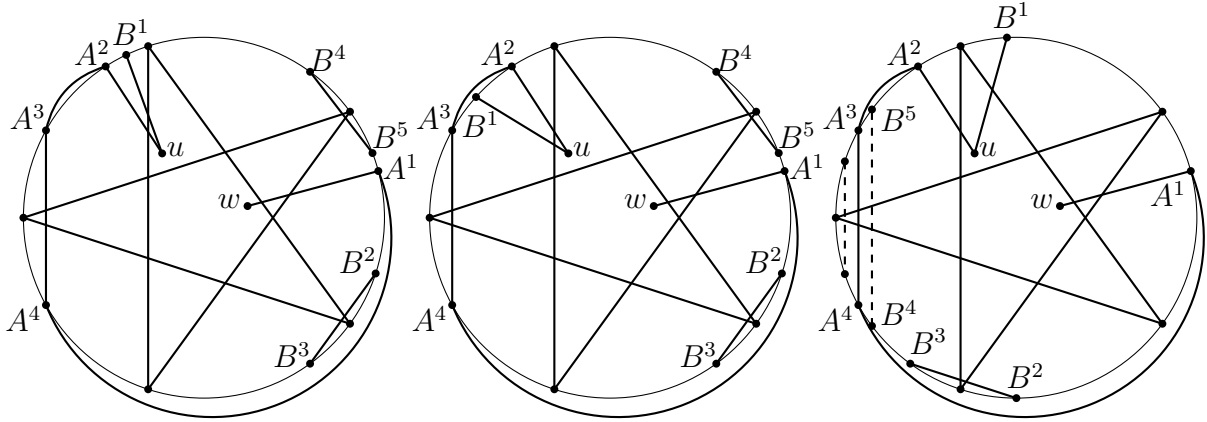

Similarly, adding the diagram of $uv$ to the diagram of $wu$ on the right in Figure~\ref{figure:diagrams5}, we get one of the diagrams of Figure \ref{figure:5uvw2}.

\begin{figure}[h]
\begin{tikzpicture}[scale=0.6]
\foreach \x in {0,1,...,4} {
\fill  ({4*cos((2*pi*\x/5+pi/5) r)},{4*sin((2*pi*\x/5+pi/5) r)}) circle (2.5pt);
\draw[thick] ({4*cos((2*pi*\x/5+pi/5) r)},{4*sin((2*pi*\x/5+pi/5) r)}) -- ({4*cos((2*pi*(\x+2)/5+pi/5) r)},{4*sin((2*pi*(\x+2)/5+pi/5) r)});
\fill  ({12+4*cos((2*pi*\x/5+pi/5) r)},{4*sin((2*pi*\x/5+pi/5) r)}) circle (2.5pt);
\draw[thick] ({12+4*cos((2*pi*\x/5+pi/5) r)},{4*sin((2*pi*\x/5+pi/5) r)}) -- ({12+4*cos((2*pi*(\x+2)/5+pi/5) r)},{4*sin((2*pi*(\x+2)/5+pi/5) r)});
}
\draw (0,0) circle (4);
\coordinate (w) at  ({cos((pi/12) r)},{sin((pi/12) r)});
\fill  (w) circle (2.5pt);
\draw ($ (w) + (-0.4,0.1) $) node {$w$};
\coordinate (A1) at ({4*cos((pi/12) r)},{4*sin((pi/12) r)});
\fill  (A1) circle (2.5pt);
\draw ($ (A1) + (0.5,0.1) $) node {$A^1$};
\draw [thick] (w)--(A1);
\coordinate (u) at  ({1.9*cos((3*pi/5+pi/12) r)},{1.9*sin((3*pi/5+pi/12) r)});
\fill  (u) circle (2.5pt);
\draw ($ (u) + (0.3,0.1) $) node {$u$};
\coordinate (A2) at ({-3.5},{-sqrt(16-3.5^2)});
\fill  (A2) circle (2.5pt);
\draw ($ (A2) + (-0.6,0.1) $) node {$A^2$};
\draw [thick] (u)--(A2);
\draw[thick] let \p1 = ($ 0.5*(A2) - 0.5*(A1) $),
\n1 = {veclen(\x1,\y1)}, \n3={atan(\y1/\x1)}
in (A1) arc (\n3:\n3-180:\n1);
\coordinate (B1) at ({4*cos((3*pi/5+pi/24) r)},{4*sin((3*pi/5+pi/24) r)});
\fill  (B1) circle (2.5pt);
\draw ($ (B1) + (0.1,0.5) $) node {$B^1$};
\draw [thick] (u)--(B1);
\coordinate (B2) at ({4*cos((-pi/5+pi/10) r)},{4*sin((-pi/5+pi/10) r)});
\fill  (B2) circle (2.5pt);
\draw ($ (B2) + (-0.4,0.3) $) node {$B^2$};
\coordinate (B3) at ({4*cos((-pi/5-pi/10) r)},{4*sin((-pi/5-pi/10) r)});
\fill  (B3) circle (2.5pt);
\draw ($ (B3) + (-0.4,0.2) $) node {$B^3$};
\draw [thick] (B3)--(B2);
\coordinate (B4) at ({4*cos((pi/5+pi/10) r)},{4*sin((pi/5+pi/10) r)});
\fill  (B4) circle (2.5pt);
\draw ($ (B4) + (0.4,0.3) $) node {$B^4$};
\coordinate (B5) at ({4*cos((pi/5-pi/12) r)},{4*sin((pi/5-pi/12) r)});
\fill  (B5) circle (2.5pt);
\draw ($ (B5) + (0.4,0.4) $) node {$B^5$};
\draw [thick] (B4)--(B5);
\draw (12,0) circle (4);
\coordinate (w) at  ({12+cos((pi/12) r)},{sin((pi/12) r)});
\fill  (w) circle (2.5pt);
\draw ($ (w) + (-0.4,0.1) $) node {$w$};
\coordinate (A1) at ({12+4*cos((pi/12) r)},{4*sin((pi/12) r)});
\fill  (A1) circle (2.5pt);
\draw ($ (A1) + (0.5,0.1) $) node {$A^1$};
\draw [thick] (w)--(A1);
\coordinate (u) at  ({12+1.9*cos((3*pi/5+pi/12) r)},{1.9*sin((3*pi/5+pi/12) r)});
\fill  (u) circle (2.5pt);
\draw ($ (u) + (0.3,0.1) $) node {$u$};
\coordinate (A2) at ({12-3.5},{-sqrt(16-3.5^2)});
\fill  (A2) circle (2.5pt);
\draw ($ (A2) + (-0.6,-0.1) $) node {$A^2$};
\draw [thick] (u)--(A2);
\draw[thick] let \p1 = ($ 0.5*(A2) - 0.5*(A1) $),
\n1 = {veclen(\x1,\y1)}, \n3={atan(\y1/\x1)}
in (A1) arc (\n3:\n3-180:\n1);
\coordinate (B1) at ({12+4*cos((3*pi/5-pi/12) r)},{4*sin((3*pi/5-pi/12) r)});
\fill  (B1) circle (2.5pt);
\draw ($ (B1) + (0.1,0.5) $) node {$B^1$};
\draw [thick] (u)--(B1);
\coordinate (B2) at ({12+4*cos((-3*pi/5+pi/10) r)},{4*sin((-3*pi/5+pi/10) r)});
\fill  (B2) circle (2.5pt);
\draw ($ (B2) + (0.2,0.5) $) node {$B^2$};
\coordinate (B3) at ({12+4*cos((-3*pi/5-pi/10) r)},{4*sin((-3*pi/5-pi/10) r)});
\fill  (B3) circle (2.5pt);
\draw ($ (B3) + (0.4,0.5) $) node {$B^3$};
\draw [thick] (B3)--(B2);
\coordinate (B5) at (12-3.8,{sqrt(16-3.8^2)});
\fill  (B5) circle (2.5pt);
\draw ($ (B5) + (-0.6,0.1) $) node {$B^5$};
\coordinate (B4) at (12-3.8,{-sqrt(16-3.8^2)});
\fill  (B4) circle (2.5pt);
\draw ($ (B4) + (-0.6,0.3) $) node {$B^4$};
\draw [thick] (B5)--(B4);
\end{tikzpicture}
\caption{Adding the diagram of $uv$. Case 2.}
\label{figure:5uvw2}
\end{figure}
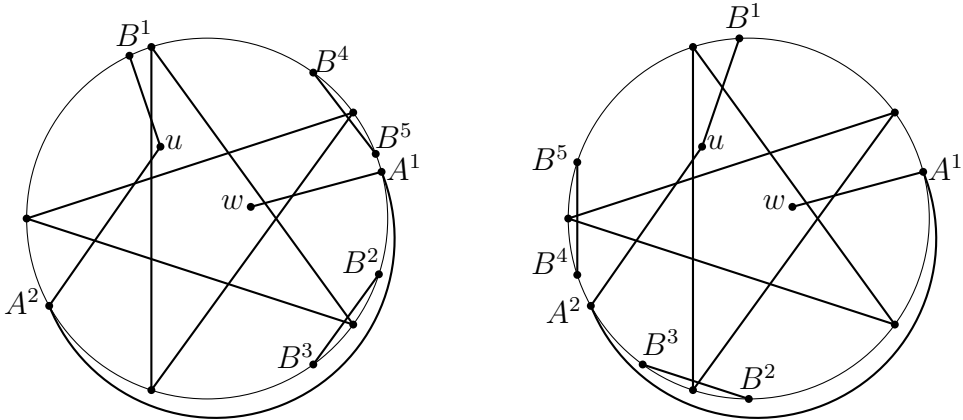

Similar to the above, in the diagram on the left in Figure~\ref{figure:5uvw2}, the point $B^1$ can only be connected to $B^4$, and then the only way to make a thrackle pass is again to join $B^4$ to $B^3$ inside $D$ going around the point $w$. But again this path cannot be deformed to one attached to a vertex of $\T(c)$. In the diagram on the right in Figure~\ref{figure:5uvw2}, the segment $B^2B^3$ is unreachable from the point $B^1$.

This completes the proof of the lemma for the case where $c$ is a 5-cycle. Now suppose that $c$ is a 7-cycle.
We start again from the edge $wu$. Fixing one of the domains labelled $2$ for $w$, there are, up to reflection, four choices of a domain labelled $1$ for $u$, as shown in the two drawings in Figure~\ref{figure:wu}. For each of the two choices for $u$ on the left in Figure~\ref{figure:wu}, edge removal may be performed on one of two edges drawn in bold; one chooses an edge that is not incident to $v$. This reduces the length of $c$ to five. For the two choices for $u$ on the right in Figure~\ref{figure:wu}, edge removal may not be possible, depending on which vertex of $\T(c)$ is $v$. It remains to treat these two possibilities.

\begin{figure}[h]
\begin{tikzpicture}[scale=0.65]
\foreach \x in {1,2} {
\fill  ({4*cos((2*pi*\x/7+pi/7) r)},{4*sin((2*pi*\x/7+pi/7) r)}) circle (2.5pt);
\draw[ultra thick] ({4*cos((2*pi*\x/7+pi/7) r)},{4*sin((2*pi*\x/7+pi/7) r)}) -- ({4*cos((2*pi*(\x+3)/7+pi/7) r)},{4*sin((2*pi*(\x+3)/7+pi/7) r)});
}
\foreach \x in {0,3,4,5,6} {
\fill  ({4*cos((2*pi*\x/7+pi/7) r)},{4*sin((2*pi*\x/7+pi/7) r)}) circle (2.5pt);
\draw[thick] ({4*cos((2*pi*\x/7+pi/7) r)},{4*sin((2*pi*\x/7+pi/7) r)}) -- ({4*cos((2*pi*(\x+3)/7+pi/7) r)},{4*sin((2*pi*(\x+3)/7+pi/7) r)});
}
\foreach \x in {0,1,...,6} {
\fill  ({12+4*cos((2*pi*\x/7+pi/7) r)},{4*sin((2*pi*\x/7+pi/7) r)}) circle (2.5pt);
\draw[thick] ({12+4*cos((2*pi*\x/7+pi/7) r)},{4*sin((2*pi*\x/7+pi/7) r)}) -- ({12+4*cos((2*pi*(\x+3)/7+pi/7) r)},{4*sin((2*pi*(\x+3)/7+pi/7) r)});
}
\coordinate (w) at  (1.1,0);
\fill  (w) circle (2.5pt);
\draw ($ (w) + (-0.45,0.1) $) node {$w$};
\coordinate (u) at  ({2*cos((pi/7) r)},{2*sin((pi/7) r)});
\fill  (u) circle (2.5pt);
\draw ($ (u) + (-0.3,0.1) $) node {$u$};
\coordinate (u) at  (-2,0);
\fill  (u) circle (2.5pt);
\draw ($ (u) + (-0.3,0.1) $) node {$u$};
\coordinate (w) at  (13.1,0);
\fill  (w) circle (2.5pt);
\draw ($ (w) + (-0.5,0.1) $) node {$w$};
\coordinate (u) at  ({12+2*cos((3*pi/7) r)},{2*sin((3*pi/7) r)});
\fill  (u) circle (2.5pt);
\draw ($ (u) + (-0.3,-0.2) $) node {$u$};
\coordinate (u) at  ({12+2*cos((5*pi/7) r)},{2*sin((5*pi/7) r)});
\fill  (u) circle (2.5pt);
\draw ($ (u) + (-0.3,0.1) $) node {$u$};
\end{tikzpicture}
\caption{Positions of $w$ and $u$.}
\label{figure:wu}
\end{figure}
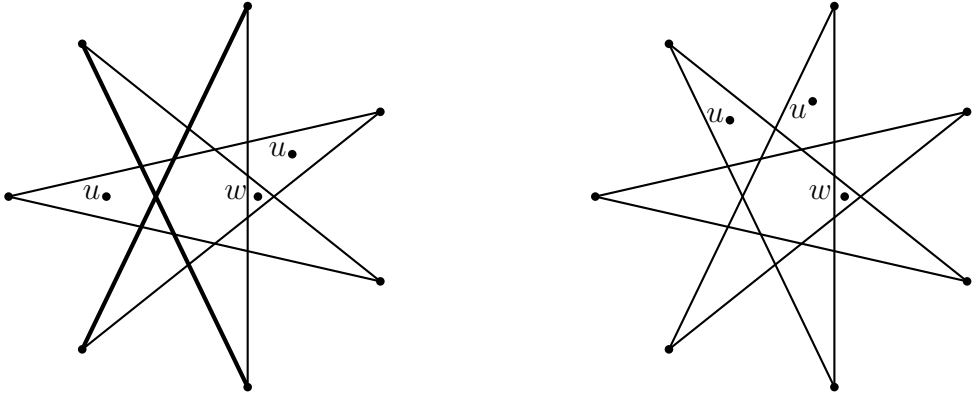

Our arguments will be similar to those above for the $5$-musquash. We place the vertices of $\T(c)$ at the $7^{\text{th}}$ roots of $-1$ and label the edges as in the proof of Lemma~\ref{lemma:o1o0}. We start with edge $wu$ and can assume that $w$ lies in the angle $\arg w \in (0, \pi/7)$. The set of edges of $\T(c)$ which the starting segment $wA^1, \; A^1 \in C$, of $wu$ crosses depends on the position of the point $A^1$ and is encoded in the vector $V(w)$. Similarly, the set of edges of $\T(c)$ crossed by $uA^2, \; A^2 \in C$, depends on the position of $A^2$ and is given by the vector $V(u)$. This time, $u$ belongs to one of the two fixed domains, as on the right in Figure~\ref{figure:wu}, and so no cyclic permutation in $V(u)$ is possible. From Table~\ref{table:VO} with $n=7$ and $s=1, 2$ we obtain the possible values for $V(w)$ and $V(u)$ shown in Table \ref{table:7wu}, where the two tables for $V(u)$ correspond to the two choices of a domain for $u$ and are already cyclically permuted accordingly.
\begin{table}[h]
\begin{center}
\begin{tabular}{|c|c|}
             \hline
              & $V(w)$ \\
             \hline
             1 & (0,0,1,0,0,0,1) \\
             2 & (1,0,1,0,0,0,0) \\
             3 & (1,0,1,0,1,1,0) \\
             4 & (1,0,0,1,1,1,0) \\
             5 & (0,1,0,1,1,1,0) \\
             6 & (0,1,0,1,1,0,1) \\
             7 & (0,1,0,0,0,0,1) \\
             \hline
\end{tabular}
\qquad
\begin{tabular}{|c|c|}
             \hline
              & $V(u)$ \\
             \hline
             $a$ & (0,0,0,0,1,0,0) \\
             $b$ & (0,0,0,0,0,1,0) \\
             $c$ & (0,0,1,1,0,1,0) \\
             $d$ & (1,1,1,1,0,1,0) \\
             $e$ & (1,1,1,1,0,0,1) \\
             $f$ & (1,1,1,0,1,0,1) \\
             $g$ & (1,0,0,0,1,0,1) \\
             \hline
\end{tabular}
\qquad
\begin{tabular}{|c|c|}
             \hline
              & $V(u)$ \\
             \hline
             $\alpha$ & (0,0,1,0,0,0,0) \\
             $\beta$ & (0,0,0,1,0,0,0) \\
             $\gamma$ & (1,1,0,1,0,0,0) \\
             $\delta$ & (1,1,0,1,0,1,1) \\
             $\varepsilon$ & (1,1,0,0,1,1,1) \\
             $\zeta$ & (1,0,1,0,1,1,1) \\
             $\eta$ & (0,0,1,0,1,1,0) \\
             \hline
\end{tabular}
\end{center}
\caption{\ }
\label{table:7wu}
\end{table}

Similar to the above, the sets of edges crossed by $wA^1$ and by $uA^2$ must be disjoint, and the complement to their union must be the union of even paths in $c$. Inspecting Table~\ref{table:7wu} we find that only the following four combinations are possible: $1b, 1\beta, 1\gamma, 5\alpha$. The corresponding diagrams are given in Figure~\ref{figure:diagrams7}, where in all the cases but one we also added the outside arcs, as they are determined uniquely by the diagrams.

\begin{figure}[h]
\begin{tikzpicture}[scale=0.7]
\foreach \x in {0,1,...,6} {
\fill  ({4*cos((2*pi*\x/7+pi/7) r)},{4*sin((2*pi*\x/7+pi/7) r)}) circle (2.5pt);
\draw[thick] ({4*cos((2*pi*\x/7+pi/7) r)},{4*sin((2*pi*\x/7+pi/7) r)}) -- ({4*cos((2*pi*(\x+3)/7+pi/7) r)},{4*sin((2*pi*(\x+3)/7+pi/7) r)});
\fill  ({12+4*cos((2*pi*\x/7+pi/7) r)},{4*sin((2*pi*\x/7+pi/7) r)}) circle (2.5pt);
\draw[thick] ({12+4*cos((2*pi*\x/7+pi/7) r)},{4*sin((2*pi*\x/7+pi/7) r)}) -- ({12+4*cos((2*pi*(\x+3)/7+pi/7) r)},{4*sin((2*pi*(\x+3)/7+pi/7) r)});
\fill  ({4*cos((2*pi*\x/7+pi/7) r)},{-10+4*sin((2*pi*\x/7+pi/7) r)}) circle (2.5pt);
\draw[thick] ({4*cos((2*pi*\x/7+pi/7) r)},{-10+4*sin((2*pi*\x/7+pi/7) r)}) -- ({4*cos((2*pi*(\x+3)/7+pi/7) r)},{-10+4*sin((2*pi*(\x+3)/7+pi/7) r)});
\fill  ({12+4*cos((2*pi*\x/7+pi/7) r)},{-10+4*sin((2*pi*\x/7+pi/7) r)}) circle (2.5pt);
\draw[thick] ({12+4*cos((2*pi*\x/7+pi/7) r)},{-10+4*sin((2*pi*\x/7+pi/7) r)}) -- ({12+4*cos((2*pi*(\x+3)/7+pi/7) r)},{-10+4*sin((2*pi*(\x+3)/7+pi/7) r)});
}
%
%
\draw (0,0) circle (4);
\coordinate (w) at  (1.1,0.1); \fill  (w) circle (2.5pt); \draw ($ (w) + (-0.45,0.1) $) node {$w$};
\coordinate (A1) at  ({1.1*4/(sqrt(1.1^2+0.1^2))},{0.1*4/(sqrt(1.1^2+0.1^2))}); \fill  (A1) circle (2.5pt); \draw ($ (A1) + (0.5,0.1) $) node {$A^1$};
\draw[thick] (w)--(A1);
\coordinate (u) at  ({1.6*cos((3*pi/7+pi/14) r)},{1.6*sin((3*pi/7+pi/14) r)}); \fill  (u) circle (2.5pt); \draw ($ (u) + (0.3,-0.2) $) node {$u$};
\coordinate (A2) at  ({4*cos((3*pi/7+pi/14) r)},{4*sin((3*pi/7+pi/14) r)}); \fill  (A2) circle (2.5pt); \draw ($ (A2) + (-0.1,0.4) $) node {$A^2$};
\draw[thick] (u)--(A2);
\coordinate (A3) at  ({4*cos((-3*pi/7+pi/10) r)},{4*sin((-3*pi/7+pi/10) r)}); \fill  (A3) circle (2.5pt); \draw ($ (A3) + (-0.1,-0.4) $) node {$A^3$};
\coordinate (A4) at  ({4*cos((-3*pi/7-pi/10) r)},{4*sin((-3*pi/7-pi/10) r)}); \fill  (A4) circle (2.5pt); \draw ($ (A4) + (-0.1,-0.4) $) node {$A^4$};
\draw[thick] (A3)--(A4);
\coordinate (A5) at  ({-4*cos((pi/10) r)},{-4*sin((pi/10) r)}); \fill  (A5) circle (2.5pt); \draw ($ (A5) + (-0.4,0.1) $) node {$A^5$};
\coordinate (A6) at  ({-4*cos((pi/10) r)},{4*sin((pi/10) r)}); \fill  (A6) circle (2.5pt); \draw ($ (A6) + (-0.4,0.1) $) node {$A^6$};
\draw[thick] (A5)--(A6);
\draw[thick] (A1) to [bend left=50] (A3);
\draw[thick] (A4) to [bend left=50] (A5);
\draw[thick] (A6) to [bend left=50] (A2);
%
%
\draw (12,0) circle (4);
\coordinate (w) at  (13.1,0.1); \fill  (w) circle (2.5pt); \draw ($ (w) + (-0.5,0.1) $) node {$w$};
\coordinate (A1) at  ({12+1.1*4/(sqrt(1.1^2+0.1^2))},{0.1*4/(sqrt(1.1^2+0.1^2))}); \fill  (A1) circle (2.5pt); \draw ($ (A1) + (0.5,0.1) $) node {$A^1$};
\draw[thick] (w)--(A1);
\coordinate (u) at  ({12+1.6*cos((5*pi/7+pi/14) r)},{1.6*sin((5*pi/7+pi/14) r)}); \fill  (u) circle (2.5pt); \draw ($ (u) + (0.1,0.3) $) node {$u$};
\coordinate (A2) at  ({12+4*cos((5*pi/7+pi/14) r)},{4*sin((5*pi/7+pi/14) r)}); \fill  (A2) circle (2.5pt); \draw ($ (A2) + (-0.1,0.4) $) node {$A^2$};
\draw[thick] (u)--(A2);
\coordinate (A3) at  ({12+4*cos((3*pi/7+pi/10) r)},{4*sin((3*pi/7+pi/10) r)}); \fill  (A3) circle (2.5pt); \draw ($ (A3) + (-0.1,-0.4) $) node {$A^3$};
\coordinate (A4) at  ({12+4*cos((3*pi/7-pi/10) r)},{4*sin((3*pi/7-pi/10) r)}); \fill  (A4) circle (2.5pt); \draw ($ (A4) + (-0.1,-0.4) $) node {$A^4$};
\draw[thick] (A3)--(A4);
\coordinate (A5) at  ({12-4*cos((pi/10) r)},{-4*sin((pi/10) r)}); \fill  (A5) circle (2.5pt); \draw ($ (A5) + (-0.4,0.1) $) node {$A^5$};
\coordinate (A6) at  ({12-4*cos((pi/10) r)},{4*sin((pi/10) r)}); \fill  (A6) circle (2.5pt); \draw ($ (A6) + (-0.4,0.1) $) node {$A^6$};
\draw[thick] (A5)--(A6);
\coordinate (v') at  ({12+4*cos((3*pi/7) r)},{4*sin((3*pi/7) r)}); \draw ($ (v') + (0.3,0.3) $) node {$v'$};
%
%
\draw (0,-10) circle (4);
\coordinate (w) at  (1.1,-10+0.1); \fill  (w) circle (2.5pt); \draw ($ (w) + (-0.45,0.1) $) node {$w$};
\coordinate (A1) at  ({1.1*4/(sqrt(1.1^2+0.1^2))},{-10+0.1*4/(sqrt(1.1^2+0.1^2))}); \fill  (A1) circle (2.5pt); \draw ($ (A1) + (0.5,0.1) $) node {$A^1$};
\draw[thick] (w)--(A1);
\coordinate (u) at  ({1.6*cos((5*pi/7+pi/14) r)},{-10+1.6*sin((5*pi/7+pi/14) r)}); \fill  (u) circle (2.5pt); \draw ($ (u) + (0.1,0.3) $) node {$u$};
\coordinate (A2) at  ({-4*cos((pi/10) r)},{-10-4*sin((pi/10) r)}); \fill  (A2) circle (2.5pt); \draw ($ (A2) + (-0.4,-0.5) $) node {$A^2$};
\draw[thick] (u)--(A2);
\coordinate (A3) at  ({4*cos((3*pi/7+pi/10) r)},{-10+4*sin((3*pi/7+pi/10) r)}); \fill  (A3) circle (2.5pt); \draw ($ (A3) + (-0.1,-0.4) $) node {$A^3$};
\coordinate (A4) at  ({4*cos((3*pi/7-pi/10) r)},{-10+4*sin((3*pi/7-pi/10) r)}); \fill  (A4) circle (2.5pt); \draw ($ (A4) + (-0.1,-0.4) $) node {$A^4$};
\draw[thick] (A3)--(A4);
\draw[thick] (A1) to [bend right=40] (A4);
\draw[thick] (A3) to [bend right=70] (A2);
%
%
\draw (12,-10) circle (4);
\coordinate (w) at  (12+1.1,-10+0.1); \fill  (w) circle (2.5pt); \draw ($ (w) + (-0.45,0.1) $) node {$w$};
\coordinate (A1) at ({12-4*cos((pi/10) r)},{-10-4*sin((pi/10) r)}); \fill  (A1) circle (2.5pt); \draw ($ (A1) + (-0.5,0.1) $) node {$A^1$};
\draw[thick] (w)--(A1);
\coordinate (u) at  ({12+1.6*cos((5*pi/7-pi/14) r)},{-10+1.6*sin((5*pi/7-pi/14) r)}); \fill  (u) circle (2.5pt); \draw ($ (u) + (-0.3,-0.2) $) node {$u$};
\coordinate (A2) at  ({12+4*cos((5*pi/7-pi/14) r)},{-10+4*sin((5*pi/7-pi/14) r)}); \fill  (A2) circle (2.5pt); \draw ($ (A2) + (-0.1,0.4) $) node {$A^2$};
\draw[thick] (u)--(A2);
\coordinate (A3) at  ({12+4*cos((pi/7+pi/10) r)},{-10+4*sin((pi/7+pi/10) r)}); \fill  (A3) circle (2.5pt); \draw ($ (A3) + (-0.5,-0.4) $) node {$A^3$};
\coordinate (A4) at  ({12+4*cos((pi/7-pi/10) r)},{-10+4*sin((pi/7-pi/10) r)}); \fill  (A4) circle (2.5pt); \draw ($ (A4) + (-0.4,-0.3) $) node {$A^4$};
\draw[thick] (A3)--(A4);
\draw[thick] let \p1 = ($ 0.5*(A4) - 0.5*(A1) $), \n1 = {veclen(\x1,\y1)}, \n3={atan(\y1/\x1)} in (A4) arc (\n3:\n3-180:\n1);
\draw[thick] (A3) to [bend right=50] (A2);
\end{tikzpicture}
\caption{Diagrams of $wu$ for the $7$-musquash: $1b$ (top left), $1\beta$ (top right), $1\gamma$ (bottom left) and $5\alpha$ (bottom right).}
\label{figure:diagrams7}
\end{figure}
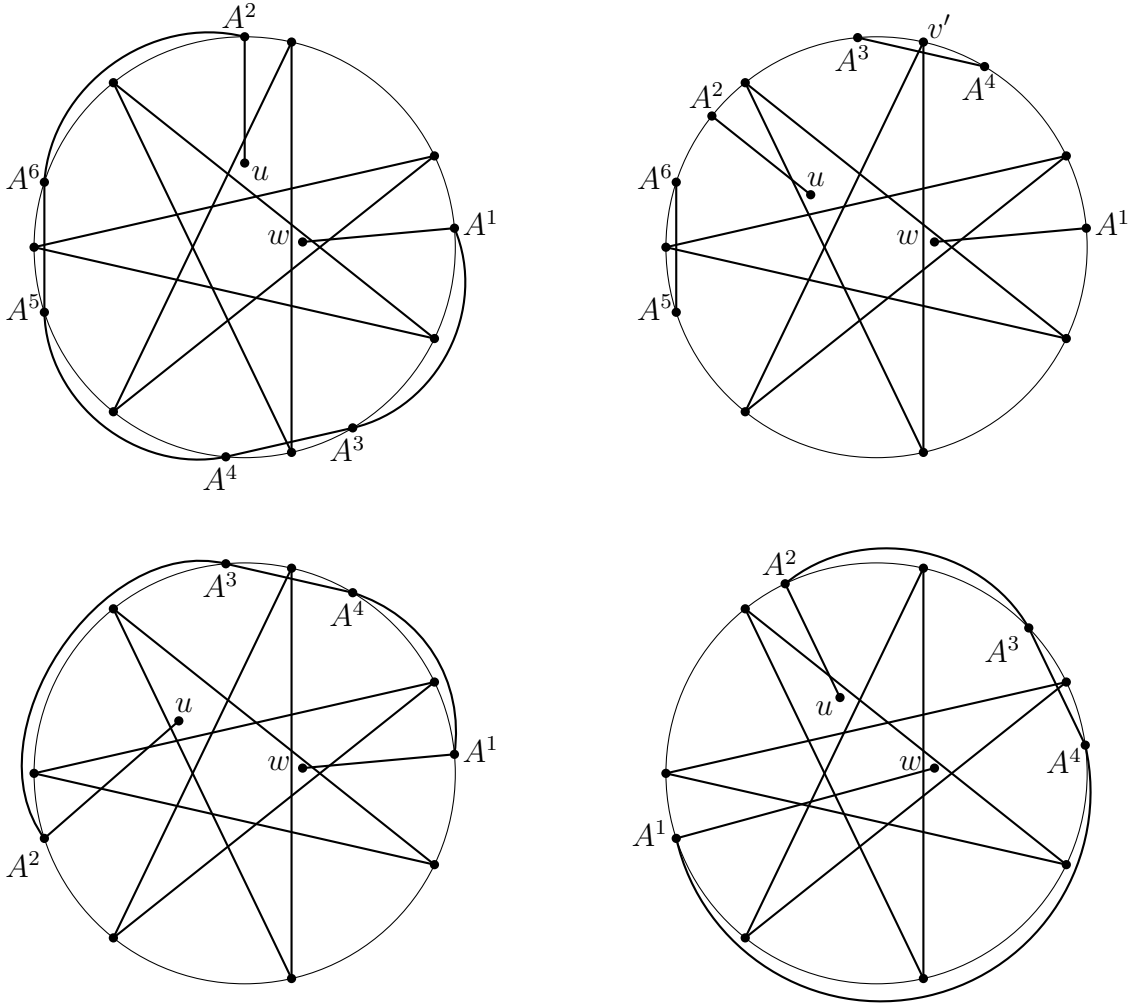

We now consider the edge $uv$. As in the case of the $5$-musquash, we slightly perturb the drawing in a neighbourhood of $v$ to move $v$ into the outer domain. Then the first two edges of the musquash $\T(c)$ which $uv$ crosses, counting from $v$, must be consecutive edges of $\T(c)$ and the starting segment $uB^1, \; B^1 \in C$, of the diagram of $uv$ is such that the complement to the set of edges of $\T(c)$ which it crosses is a union of even paths. Consulting Table~\ref{table:7wu} we find that this only happens in the cases $a, b, e, \alpha, \beta$ and $\varepsilon$. The first three apply to the top left diagram in Figure~\ref{figure:diagrams7}. The choices of $a$, of $e$ and of $b$, in the case when $uB^1$ goes to the left of $uA^1$, result in some of the edges of $\T(c)$ becoming unreachable for $uv$. With the choice of $b$, if $uB^1$ goes to the right of $uA^2$, it can be completed to a thrackle path, but it violates the condition that the last two crossings on $uv$ counting from $u$ occur with the adjacent edges of $\T(c)$. The cases $\alpha, \beta$ and $\varepsilon$ apply to the other three diagrams in Figure~\ref{figure:diagrams7}. For the bottom right diagram, the choices of $\beta, \varepsilon$ and of $\alpha$, when $uB^1$ goes to the left of $uA^2$, again result in unreachable edges. If we choose $\alpha$ with $uB^1$ going to the right of $uA^2$, then the diagram for $uv$ contains a small straight line segment with the endpoints on $C$ in a neighbourhood of the vertex of $\T(c)$ next to $A^1$ in the negative direction, which cannot be reached. Similarly, for the bottom left diagram, the choices of $\varepsilon$ and $\alpha$ result in unreachable edges. If we choose $\beta$, then considering the diagram for $uv$ we find that it can be completed to a thrackled path, but the condition that the last two crossings on $uv$ counting from $u$ occur with the adjacent edges of $\T(c)$ is violated.

The only remaining diagram is the one on the top right in Figure~\ref{figure:diagrams7}. We can again use the idea of edge removal. By Lemma~\ref{lemma:triangle}, the points $u$ and $w$ prohibit edge removal on all the edges except for the two adjacent to the vertex $v'$. It follows that, unless the common vertex of the cycle $c$ and the two-path $p$ is $v'$, we can perform the edge removal operation, which leads to the case of the $5$-musquash that has already been considered. It only remains to consider the case when $c$ and $p$ share the vertex $v'$ (or equivalently, after perturbing it, the diagram of the resulting path $uv$ contains a small straight line segment with the endpoints on $C$ in a neighbourhood of $v'$). This immediately prohibits the choice of $\varepsilon$ and also the choice of $\alpha$ for the starting segment $uB^1$, since the resulting diagram for $uv$ does not contain the required segment. For $\beta$ we have two possibilities: the starting segment $uB^1$ can go to the left or to the right of $uA^2$. Suppose it goes to the left, so that $B^1$ lies on $C$ between $A^2$ and $A^6$. Then $wu$ cannot contain $A^2A^6$ as one of its outside arcs, which means that the outside arcs are $A^1A^5, A^6A^4$ and $A^3A^2$, which also leads to a contradiction: $uB^1$ cannot be extended. A similar argument also works when $uB^1$ goes the right $uA^2$ (so that $B^1$ lies on $C$ between $A^2$ and the vertex of $\T(c)$ next to $A^2$ in the negative direction). In that case $wu$ cannot contain an outside arc $A^2A^3$, and so the outside arcs are $A^1A^4, A^3A^5$ and $A^6A^2$. Then, again, $uB^1$ cannot be extended to the correct diagram of $uv$.
\end{proof}

\begin{proof}[Proof of Lemma~\ref{lemma:no55}]  Let $G$ be a figure-eight graph comprised of two five-cycles $c_1,c_2$ sharing a vertex $v$, and assume that there exists a thrackle drawing $\T(G)$. Any thrackle drawing of a five-cycle is a standard musquash, and by \cite[Lemma~2.2]{LPS97}, the drawings of $c_1$ and $c_2$ cross at $v$. Fix a drawing of $c_1$, with the vertices $v,1,2,3,4$  at the vertices of a regular pentagon, with $v$ at the top, as shown in Figure \ref{figure:19}. We label $v,v_1,v_2,v_3,v_4$ the vertices of $c_2$ in consecutive order in such a way that the starting segment of $\T(vv_1)$ lies in the outer domain of the complement to $\T(c_1)$.

Label $0, 1, 2$ the domains of the complement of $\T(c_1)$ as in Section~\ref{section:th}. Then by Lemma~\ref{lemma:1edge} and Lemma~\ref{lemma:2path}, the vertices $v_2$ and $v_4$ of $\T(c_2)$ lie in the outer domain of $\T(c_1)$, and the vertices $v_1$ and $v_3$ lie in domains labelled by $1$.

If edge removal on one of the edges of $G$ is possible, we immediately arrive at a contradiction, as the resulting drawing is a thrackled figure-eight graph comprised of a three-cycle and a five-cycle, which can be further reduced by edge removal to a figure-eight graph consisting of two three-cycles, as explained at the end of Section~\ref{section:th}. The arguments which we used there (based on Lemma~\ref{lemma:triangle}) show that, up to a reflection, there is only two possible positions of the vertices $v_1$, $v_3$ for which no edges of $\T(c_1)$ can be removed, as in Figure~\ref{figure:18}.

\begin{figure}[h]
\begin{tikzpicture}[scale=0.6]
\foreach \x in {0,1,...,4} {
\fill  ({4*cos((2*pi*\x/5+pi/2) r)},{4*sin((2*pi*\x/5+pi/2) r)}) circle (2.5pt);
\draw[thick] ({4*cos((2*pi*\x/5+pi/2) r)},{4*sin((2*pi*\x/5+pi/2) r)}) -- ({4*cos((2*pi*(\x+2)/5+pi/2) r)},{4*sin((2*pi*(\x+2)/5+pi/2) r)});
\fill  ({10+4*cos((2*pi*\x/5+pi/2) r)},{4*sin((2*pi*\x/5+pi/2) r)}) circle (2.5pt);
\draw[thick] ({10+4*cos((2*pi*\x/5+pi/2) r)},{4*sin((2*pi*\x/5+pi/2) r)}) -- ({10+4*cos((2*pi*(\x+2)/5+pi/2) r)},{4*sin((2*pi*(\x+2)/5+pi/2) r)});
}
\fill  ({2*cos((2*pi/5+pi/2) r)},{2*sin((2*pi/5+pi/2) r)}) circle (2.5pt);
\fill  ({2*cos((-2*pi/5+pi/2) r)},{2*sin((-2*pi/5+pi/2) r)}) circle (2.5pt);
\fill  ({10+2*cos((-2*pi*2/5+pi/2) r)},{2*sin((-2*pi*2/5+pi/2) r)}) circle (2.5pt);
\fill  ({10+2*cos((-2*pi/5+pi/2) r)},{2*sin((-2*pi/5+pi/2) r)}) circle (2.5pt);
\draw (.5,3.85) node {$v$};
\draw (10.5,3.85) node {$v$};
\end{tikzpicture}
\caption{Two possible positions of the vertices $v_1$ and $v_3$ relative to $v$.}
\label{figure:18}
\end{figure}
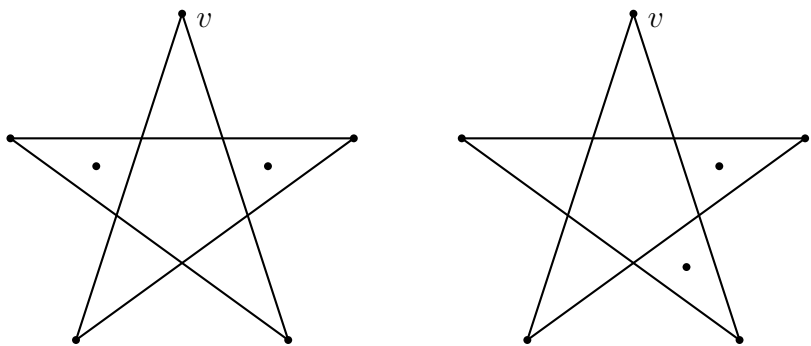

We first reduce the case on the left in Figure~\ref{figure:18} to the case on the right.
Assume that the vertex $v_1$ lies either in the upper-left or in the upper-right domain labelled $1$ of the drawing $\T(c_1)$. Up to isotopy, and a Reidemeister move of the third kind on the triple of edges $12, 34, vv_4$, there is only one possibility for the edges $vv_4$ and $vv_1$ (Figure~\ref{figure:19}).

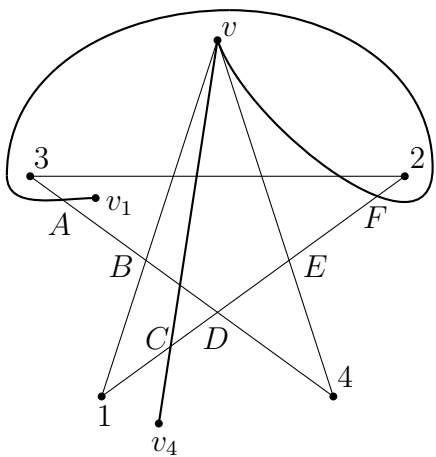
\begin{figure}[h]
\begin{tikzpicture}
\draw (4.46,-1.04)-- (6,3.7);
\draw (6,3.7)-- (7.54,-1.04);
\draw (7.54,-1.04)-- (3.51,1.89);
\draw (3.51,1.89)-- (8.49,1.89);
\draw (8.49,1.89)-- (4.46,-1.04);
\draw [thick] (6,3.7)-- (5.22,-1.4);
\fill (4.46,-1.04) circle (1.5pt);
\draw (4.5,-1.3) node {$1$};
\fill (7.54,-1.04) circle (1.5pt);
\draw (7.7,-0.78) node {$4$};
\draw (8.1,1.35) node {$F$};
\draw (7.3,0.7) node {$E$};
\fill (8.49,1.89) circle (1.5pt);
\draw (8.66,2.14) node {$2$};
\fill (6,3.7) circle (1.5pt);
\draw (6.16,3.85) node {$v$};
\draw (5.98,-0.26) node {$D$};
\draw (5.2,-0.26) node {$C$};
\draw  (4.72,0.7)  node {$B$};
\draw  (3.9,1.25)  node {$A$};
\fill (3.51,1.89) circle (1.5pt);
\draw (3.66,2.14) node {$3$};
\fill  (5.22,-1.4) circle (1.5pt);
\draw (5.3,-1.7) node {$v_4$};
\fill  (4.38,1.6) circle (1.5pt);
\draw (4.7,1.5) node {$v_1$};
\draw[thick] (4.38,1.6) to[out=180,in=-90] (3.2,1.9);
\draw[thick] (3.2,1.9) to[out=90,in=-180] (6.16,4.1);
\draw[thick] (6.16,4.1) to[out=0,in=90] (8.86,2);
\draw[thick] (8.86,2) to[out=-90,in=-70] (6,3.7);
\end{tikzpicture}
\caption{The musquash $\T(c_1)$ with two edges of $\T(c_2)$.}
\label{figure:19}
\end{figure}

We want to add the edge $v_4v_3$ to the drawing in Figure~\ref{figure:19}. The first crossing on $v_4v_3$ counting from $v_4$ cannot lie on the segment $B1$ (the edge $v5$ would be unreachable), on the segment $4E$ (the edge $v2$ would be unreachable), on the segment $FA$ of $vv_1$ (the edge $v4$ would be unreachable), on the segment $EF$ ($v1$ and $v4$ would not both be reachable), and on the segment $AB$ ($v1$ and $v4$ would not both be reachable). So the first crossing on  $ v_4v_3$ occurs on one of the segments $1C$, $CD$ or $D4$. Furthermore, the vertices $v_1$, $v_3$ can only lie in the domains shown in Figure~\ref{figure:18}, which gives (up to isotopy and Reidemeister moves) two drawings shown in Figure~\ref{figure:20}.

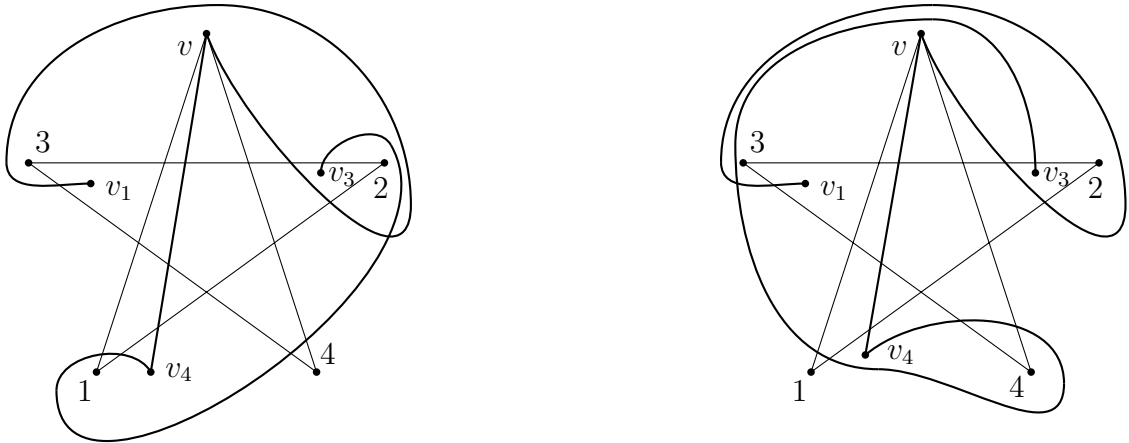
\begin{figure}[h]
\begin{tikzpicture}[scale=0.95]
%
%
\coordinate (v) at ({6},3.7); \fill (v) circle (1.5pt); \draw ($(v)+(-0.3,-0.2)$) node {$v$};
\coordinate (A2) at ({4.46},-1.04); \fill (A2) circle (1.5pt); \draw ($(A2)+(-0.16,-0.26)$) node {$1$};
\coordinate (A3) at ({8.49},1.89); \fill (A3) circle (1.5pt); \draw ($(A3)+(-0.05,-0.35)$) node {$2$};
\coordinate (A4) at ({3.51},1.89); \fill (A4) circle (1.5pt); \draw ($(A4)+(0.2,0.3)$) node {$3$};
\coordinate (A5) at ({7.54},-1.04); \fill (A5) circle (1.5pt); \draw ($(A5)+(0.16,0.26)$) node {$4$};
\coordinate (v1) at ({4.38},1.6); \fill (v1) circle (1.5pt); \draw ($(v1)+(0.4,-0.1)$) node {$v_1$};
\coordinate (v3) at ({7.6},1.75); \fill (v3) circle (1.5pt); \draw ($(v3)+(0.3,-0.05)$) node {$v_3$};
\coordinate (v4) at ({5.22},-1.04); \fill (v4) circle (1.5pt); \draw ($(v4)+(0.4,0)$) node {$v_4$};
\draw (A2)--(v)--(A5)--(A4)--(A3)--(A2);
\draw[thick] (v)-- (v4);
\draw[thick] (v1) to[out=180,in=-90] (3.2,1.9);
\draw[thick] (3.2,1.9) to[out=90,in=-180] (6.16,4.1);
\draw[thick] (6.16,4.1) to[out=0,in=90] (8.86,1.3);
\draw[thick] (8.86,1.3) to[out=-90,in=-70] (v);
\draw[thick] (v4) to[out=120,in=90] (3.9,-1.3);
\draw[thick] (3.9,-1.3) to[out=-90,in=-60] (8.58,2.14);
\draw[thick] (8.58,2.14) to[out=120,in=90] (v3);
%
%
\coordinate (v) at ({6+10},3.7); \fill (v) circle (1.5pt); \draw ($(v)+(-0.3,-0.2)$) node {$v$};
\coordinate (A2) at ({4.46+10},-1.04); \fill (A2) circle (1.5pt); \draw ($(A2)+(-0.16,-0.26)$) node {$1$};
\coordinate (A3) at ({8.49+10},1.89); \fill (A3) circle (1.5pt); \draw ($(A3)+(-0.05,-0.35)$) node {$2$};
\coordinate (A4) at ({3.51+10},1.89); \fill (A4) circle (1.5pt); \draw ($(A4)+(0.2,0.3)$) node {$3$};
\coordinate (A5) at ({7.54+10},-1.04); \fill (A5) circle (1.5pt); \draw ($(A5)+(-0.2,-0.2)$) node {$4$};
\coordinate (v1) at ({4.38+10},1.6); \fill (v1) circle (1.5pt); \draw ($(v1)+(0.4,-0.1)$) node {$v_1$};
\coordinate (v3) at ({7.6+10},1.75); \fill (v3) circle (1.5pt); \draw ($(v3)+(0.3,-0.05)$) node {$v_3$};
\coordinate (v4) at ({5.22+10},-0.8); \fill (v4) circle (1.5pt); \draw ($(v4)+(0.5,0)$) node {$v_4$};
\draw (A2)--(v)--(A5)--(A4)--(A3)--(A2);
\draw[thick] (v)-- (v4);
\draw[thick] (14.38,1.6) to[out=180,in=-90] (13.2,1.9);
\draw[thick] (13.2,1.9) to[out=90,in=-180] (16.16,4.1);
\draw[thick] (16.16,4.1) to[out=0,in=90] (18.86,1.3);
\draw[thick] (18.86,1.3) to[out=-90,in=-70] (16,3.7);
\draw[thick] (15.22,-0.8) to[out=40,in=90] (18,-1.2);
\draw[thick] (18,-1.2) to[out=-90,in=0] (15.4,-1);
\draw[thick] (15.4,-1) to[out=180,in=-90] (13.4,2.14);
\draw[thick] (13.4,2.14) to[out=90,in=180] (16.16,3.9);
\draw[thick] (16.16,3.9) to[out=0,in=90] (17.6,1.75);
\coordinate (SW) at (current bounding box.south west);
\coordinate (NE) at (current bounding box.north east);
\pgfresetboundingbox
\path[use as bounding box] ($(SW) +(0,1.5)$) -- ($(NE) +(0,0)$);
\end{tikzpicture}
\caption{The musquash $\T(c_1)$ with three edges of $\T(c_2)$.}
\label{figure:20}
\end{figure}

Consider the triangular domain $\triangle$ defined by the edge $vv_4$ of the musquash $\T(c_2)$. In both cases of Figure~\ref{figure:20}, $\triangle$ contains exactly one vertex of $c_1$, the vertex $1$ adjacent to $v$, and the starting segment of the edge $v1$ of $\T(c_1)$ lies in the outer domain relative to $\T(c_2)$.
It follows that interchanging the cycles $c_1$, $c_2$ if necessary, we can always assume that the vertex $v_1$ lies in one of the triangular domains of $\T(c_1)$ defined by an edge of $c_1$ incident to $v$. This gives the case on the right in Figure~\ref{figure:18}, and what is more, we now know that $v_1$ lies in the lower-right domain of the complement to $\T(c_1)$, and $v_3$, in the upper-right domain.

Up to Reidemeister moves, there are two ways to attach the edges $vv_1$, $vv_4$ to $\T(c_1)$ such that $v_1$ lies in the correct domain (Figure~\ref{figure:21}), and then there are three ways to further attach the edge $v_4v_3$ in such a manner that $v_3$ lies in the correct domain, as in Figure~\ref{figure:22}.

\begin{figure}[h]
\begin{minipage}[b]{6cm}
\begin{tikzpicture}[scale=0.9]
\draw (4.46,-1.04)-- (6,3.7);
\draw (6,3.7)-- (7.54,-1.04);
\draw (7.54,-1.04)-- (3.51,1.89);
\draw (3.51,1.89)-- (8.49,1.89);
\draw (8.49,1.89)-- (4.46,-1.04);
\draw[thick] (6,3.7)-- (5.22,-0.8);
\fill (4.46,-1.04) circle (1.5pt);
\fill (7.54,-1.04) circle (1.5pt);
\fill (8.49,1.89) circle (1.5pt);
\fill (6,3.7) circle (1.5pt);
\draw (6.16,3.85) node {$v$};
\fill (3.51,1.89) circle (1.5pt);
\fill  (5.22,-0.8) circle (1.5pt);
\draw (5.62,-0.85) node {$v_4$};
\fill (6.98,-0.4) circle (1.5pt);
\draw (7,-0.14) node {$v_1$};
\draw[ thick] (6,3.7) to[out=-60,in=90] (7.96,-1.3);
\draw[ thick] (7.96,-1.3) to[out=-90,in=-90] (6.98,-0.4);
\end{tikzpicture}
\end{minipage}
\hspace{1.0cm}
\begin{minipage}[b]{6cm}
\begin{tikzpicture}[scale=0.9]
\draw (4.46,-1.04)-- (6,3.7);
\draw (6,3.7)-- (7.54,-1.04);
\draw (7.54,-1.04)-- (3.51,1.89);
\draw (3.51,1.89)-- (8.49,1.89);
\draw (8.49,1.89)-- (4.46,-1.04);
\fill (4.46,-1.04) circle (1.5pt);
\fill (7.54,-1.04) circle (1.5pt);
\fill (8.49,1.89) circle (1.5pt);
\fill (6,3.7) circle (1.5pt);
\draw (6.26,3.65) node {$v$};
\fill (3.51,1.89) circle (1.5pt);
\fill  (7.12,-0.94) circle (1.5pt);
\draw (7.2,-1.2) node {$v_4$};
\fill  (6.52,0.02) circle (1.5pt);
\draw (6.6,0.28) node {$v_1$};
\draw[thick] (6,3.7)-- (7.12,-0.94);
\draw[ thick] (6,3.7) to[out=-60,in=90] (7.96,-1.4);
\draw[ thick] (7.96,-1.4) to[out=-90,in=0] (7.2,-1.4);
\draw[ thick] (7.2,-1.4) to[out=180,in=-90] (6.52,0.02);
\end{tikzpicture}
\end{minipage}
\caption{The musquash $\T(c_1)$ with two edges of $\T(c_2)$.}
\label{figure:21}
\end{figure}
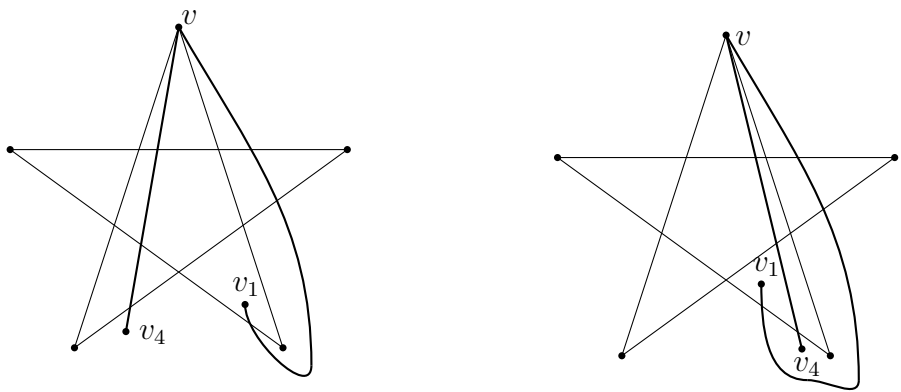


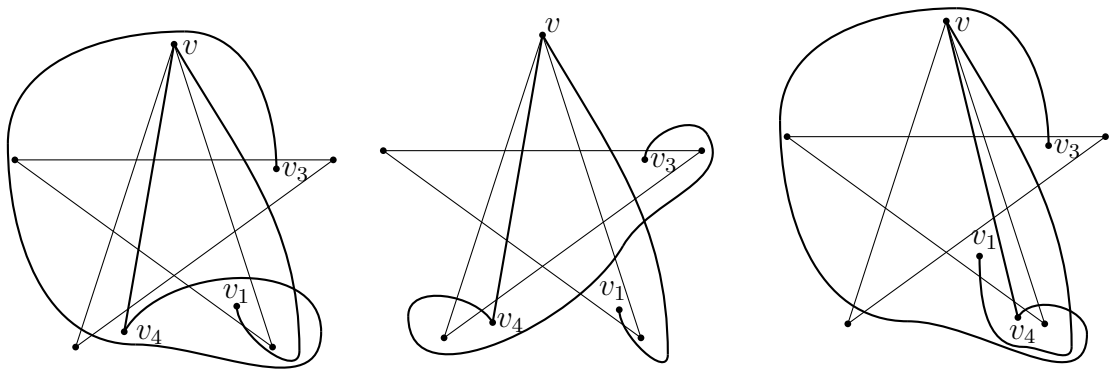
\begin{figure}[h]
\centering
\begin{tikzpicture}[scale=0.85]
\draw (4.46,-1.04)-- (6,3.7);
\draw (6,3.7)-- (7.54,-1.04);
\draw (7.54,-1.04)-- (3.51,1.89);
\draw (3.51,1.89)-- (8.49,1.89);
\draw (8.49,1.89)-- (4.46,-1.04);
\draw [thick](6,3.7)-- (5.22,-0.8);
\fill (4.46,-1.04) circle (1.5pt);
\fill (7.54,-1.04) circle (1.5pt);
\fill (8.49,1.89) circle (1.5pt);
\fill (6,3.7) circle (1.5pt);
\draw (6.26,3.65) node {$v$};
\fill (3.51,1.89) circle (1.5pt);
\fill  (5.22,-0.8) circle (1.5pt);
\draw (5.62,-0.85) node {$v_4$};
\fill  (6.98,-0.4) circle (1.5pt);
\draw (6.98,-0.2) node {$v_1$};
\fill (7.6,1.75) circle (1.5pt);
\draw (7.9,1.7) node {$v_3$};
\draw[thick] (5.22,-0.8) to[out=60,in=90] (8.3,-0.8);
\draw[thick] (8.3,-0.8) to[out=-90,in=0] (5.4,-1);
\draw[thick] (5.4,-1) to[out=180,in=-90] (3.4,2.14);
\draw[thick] (3.4,2.14) to[out=90,in=180] (6.16,3.9);
\draw[thick] (6.16,3.9) to[out=0,in=90] (7.6,1.75);
\draw[ thick] (6,3.7) to[out=-60,in=90] (7.96,-1.1);
\draw[ thick] (7.96,-1.1) to[out=-90,in=-90] (6.98,-0.4);
\end{tikzpicture}
\hspace{0.3cm}
\begin{tikzpicture}[scale=0.85]
\draw (4.46,-1.04)-- (6,3.7);
\draw (6,3.7)-- (7.54,-1.04);
\draw (7.54,-1.04)-- (3.51,1.89);
\draw (3.51,1.89)-- (8.49,1.89);
\draw (8.49,1.89)-- (4.46,-1.04);
\draw [thick] (6,3.7)-- (5.22,-0.8);
\fill (4.46,-1.04) circle (1.5pt);
\fill (7.54,-1.04) circle (1.5pt);
\fill (8.49,1.89) circle (1.5pt);
\fill (6,3.7) circle (1.5pt);
\draw (6.16,3.85) node {$v$};
\fill (3.51,1.89) circle (1.5pt);
\fill  (5.22,-0.8) circle (1.5pt);
\draw (5.5,-0.8) node {$v_4$};
\fill (7.6,1.75) circle (1.5pt);
\fill (7.2,-0.6) circle (1.5pt);
\draw (7.9,1.7) node {$v_3$};
\draw (7.1,-0.3) node {$v_1$};
\draw[thick] (5.22,-0.8) to[out=120,in=90] (3.9,-0.7);
\draw[thick] (3.9,-0.7) to[out=-90,in=-120] (7.26,0.4);
\draw[thick] (7.26,0.4) to[out=60,in=-60] (8.58,2.14);
\draw[thick] (8.58,2.14) to[out=120,in=90] (7.6,1.75);
\draw[ thick] (6,3.7) to[out=-60,in=90] (7.96,-1.3);
\draw[ thick] (7.96,-1.3) to[out=-90,in=-90] (7.2,-0.6);
\end{tikzpicture}
\hspace{0.3cm}
\begin{tikzpicture}[scale=0.85]
\draw (4.46,-1.04)-- (6,3.7);
\draw (6,3.7)-- (7.54,-1.04);
\draw (7.54,-1.04)-- (3.51,1.89);
\draw (3.51,1.89)-- (8.49,1.89);
\draw (8.49,1.89)-- (4.46,-1.04);
\draw [thick](6,3.7)-- (7.12,-0.94);
\fill (4.46,-1.04) circle (1.5pt);
\fill (7.54,-1.04) circle (1.5pt);
\fill (8.49,1.89) circle (1.5pt);
\fill (6,3.7) circle (1.5pt);
\draw (6.26,3.65) node {$v$};
\fill (3.51,1.89) circle (1.5pt);
\fill  (7.12,-0.94) circle (1.5pt);
\draw (7.2,-1.2) node {$v_4$};
\fill  (6.52,0.02) circle (1.5pt);
\draw (6.6,0.28) node {$v_1$};
\draw[ thick] (6,3.7) to[out=-60,in=90] (7.96,-1.4);
\draw[ thick] (7.96,-1.4) to[out=-90,in=0] (7.2,-1.4);
\draw[ thick] (7.2,-1.4) to[out=180,in=-90] (6.52,0.02);
\draw[thick] (7.12,-0.94) to[out=60,in=90] (8.2,-1.2);
\draw[thick] (8.2,-1.2) to[out=-90,in=0] (5.4,-1);
\draw[thick] (5.4,-1) to[out=180,in=-90] (3.4,2.14);
\draw[thick] (3.4,2.14) to[out=90,in=180] (6.16,3.9);
\draw[thick] (6.16,3.9) to[out=0,in=90] (7.6,1.75);
\fill (7.6,1.75) circle (1.5pt);
\draw (7.9,1.7) node {$v_3$};
\end{tikzpicture}
\caption{The musquash $\T(c_1)$ with three edges of $\T(c_2)$.}
\label{figure:22}
\end{figure}

Next, we attach the edge $v_1v_2$ to the drawings in Figure~\ref{figure:22}. Since $c_2$ is thrackled as a standard musquash, the resulting drawing satisfies the following condition: the crossings of the edges $vv_1$ and $v_1v_2$ with the edge $v_3v_4$ have  opposite orientation, and on $v_3v_4$ counting from $v_3$, the former crossing precedes the latter  one. This gives us a unique possible way (up to isotopy and Reidemeister moves) of attaching the edge $v_1v_2$ to each of the drawings in Figure~\ref{figure:22}. The resulting thrackles are shown in Figure~\ref{figure:23}.

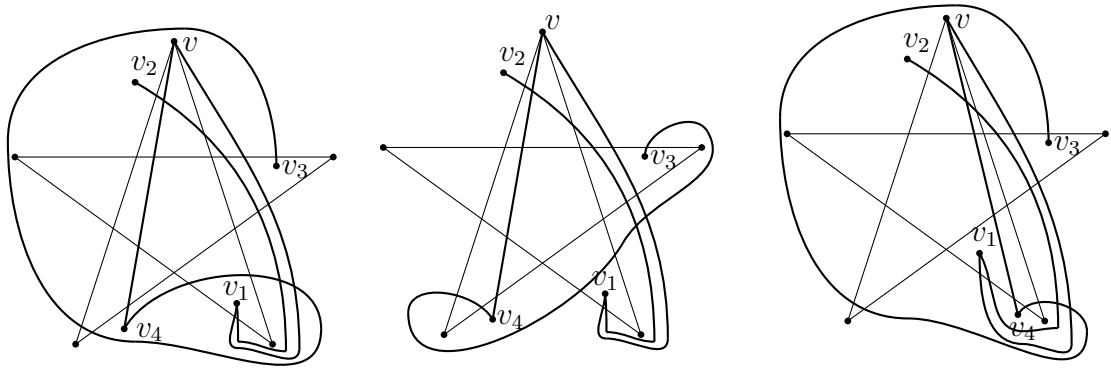
\begin{figure}[h]
\centering
\begin{tikzpicture}[scale=0.85]
\draw (4.46,-1.04)-- (6,3.7);
\draw (6,3.7)-- (7.54,-1.04);
\draw (7.54,-1.04)-- (3.51,1.89);
\draw (3.51,1.89)-- (8.49,1.89);
\draw (8.49,1.89)-- (4.46,-1.04);
\draw [thick](6,3.7)-- (5.22,-0.8);
\fill (4.46,-1.04) circle (1.5pt);
\fill (7.54,-1.04) circle (1.5pt);
\fill (8.49,1.89) circle (1.5pt);
\fill (6,3.7) circle (1.5pt);
\draw (6.26,3.65) node {$v$};
\fill (3.51,1.89) circle (1.5pt);
\fill  (5.22,-0.8) circle (1.5pt);
\draw (5.62,-0.85) node {$v_4$};
\fill  (6.98,-0.4) circle (1.5pt);
\draw (6.98,-0.2) node {$v_1$};
\fill (7.6,1.75) circle (1.5pt);
\draw (7.9,1.7) node {$v_3$};
\fill  (5.39,3.06) circle (1.5pt);
\draw (5.54,3.32) node {$v_2$};
\draw[thick] (5.22,-0.8) to[out=60,in=90] (8.3,-0.8);
\draw[thick] (8.3,-0.8) to[out=-90,in=0] (5.4,-1);
\draw[thick] (5.4,-1) to[out=180,in=-90] (3.4,2.14);
\draw[thick] (3.4,2.14) to[out=90,in=180] (6.16,3.9);
\draw[thick] (6.16,3.9) to[out=0,in=90] (7.6,1.75);
\draw[ thick] (6,3.7) to[out=-60,in=90] (7.96,-1.1);
\draw[ thick] (7.96,-1.1) to[out=-90,in=0] (7,-1.1);
\draw[ thick] (7,-1.1) to[out=180,in=-90] (6.98,-0.4);
\draw[ thick] (6.98,-0.4) to[out=-85,in=90] (7,-1);
\draw[ thick] (7,-1) to[out=0,in=180] (7.75,-1.15);
\draw[ thick] (7.75,-1.15) to[out=90,in=-30] (5.39,3.06);
\end{tikzpicture}
\hspace{0.3cm}
\begin{tikzpicture}[scale=0.85]
\draw (4.46,-1.04)-- (6,3.7);
\draw (6,3.7)-- (7.54,-1.04);
\draw (7.54,-1.04)-- (3.51,1.89);
\draw (3.51,1.89)-- (8.49,1.89);
\draw (8.49,1.89)-- (4.46,-1.04);
\draw [thick] (6,3.7)-- (5.22,-0.8);
\fill (4.46,-1.04) circle (1.5pt);
\fill (7.54,-1.04) circle (1.5pt);
\fill (8.49,1.89) circle (1.5pt);
\fill (6,3.7) circle (1.5pt);
\draw (6.16,3.85) node {$v$};
\fill (3.51,1.89) circle (1.5pt);
\fill  (5.22,-0.8) circle (1.5pt);
\draw (5.5,-0.8) node {$v_4$};
\fill (7.6,1.75) circle (1.5pt);
\draw (7.9,1.7) node {$v_3$};
\fill  (6.98,-0.4) circle (1.5pt);
\draw (6.98,-0.2) node {$v_1$};
\fill  (5.39,3.06) circle (1.5pt);
\draw (5.54,3.32) node {$v_2$};
\draw[thick] (5.22,-0.8) to[out=120,in=90] (3.9,-0.7);
\draw[thick] (3.9,-0.7) to[out=-90,in=-120] (7.26,0.4);
\draw[thick] (7.26,0.4) to[out=60,in=-60] (8.58,2.14);
\draw[thick] (8.58,2.14) to[out=120,in=90] (7.6,1.75);
\draw[ thick] (6,3.7) to[out=-60,in=90] (7.96,-1.1);
\draw[ thick] (7.96,-1.1) to[out=-90,in=0] (7,-1.1);
\draw[ thick] (7,-1.1) to[out=180,in=-90] (6.98,-0.4);
\draw[ thick] (6.98,-0.4) to[out=-85,in=90] (7,-1);
\draw[ thick] (7,-1) to[out=0,in=180] (7.75,-1.15);
\draw[ thick] (7.75,-1.15) to[out=90,in=-30] (5.39,3.06);
\end{tikzpicture}
\hspace{0.3cm}
\begin{tikzpicture}[scale=0.85]
\draw (4.46,-1.04)-- (6,3.7);
\draw (6,3.7)-- (7.54,-1.04);
\draw (7.54,-1.04)-- (3.51,1.89);
\draw (3.51,1.89)-- (8.49,1.89);
\draw (8.49,1.89)-- (4.46,-1.04);
\draw[thick] (6,3.7)-- (7.12,-0.94);
\fill (4.46,-1.04) circle (1.5pt);
\fill (7.54,-1.04) circle (1.5pt);
\fill (8.49,1.89) circle (1.5pt);
\fill (6,3.7) circle (1.5pt);
\draw (6.26,3.65) node {$v$};
\fill (3.51,1.89) circle (1.5pt);
\fill  (7.12,-0.94) circle (1.5pt);
\draw (7.2,-1.2) node {$v_4$};
\fill  (6.52,0.02) circle (1.5pt);
\draw (6.6,0.28) node {$v_1$};
\fill (7.6,1.75) circle (1.5pt);
\draw (7.9,1.7) node {$v_3$};
\fill  (5.39,3.06) circle (1.5pt);
\draw (5.54,3.32) node {$v_2$};
\draw[ thick] (6,3.7) to[out=-60,in=90] (7.96,-1.4);
\draw[ thick] (7.96,-1.4) to[out=-90,in=0] (7.2,-1.4);
\draw[ thick] (7.2,-1.4) to[out=180,in=-90] (6.52,0.02);
\draw[thick] (7.12,-0.94) to[out=60,in=90] (8.2,-1.2);
\draw[thick] (8.2,-1.2) to[out=-90,in=0] (5.4,-1);
\draw[thick] (5.4,-1) to[out=180,in=-90] (3.4,2.14);
\draw[thick] (3.4,2.14) to[out=90,in=180] (6.16,3.9);
\draw[thick] (6.16,3.9) to[out=0,in=90] (7.6,1.75);
\draw[ thick] (6.52,0.02) to[out=-60,in=180] (7.2,-1.2);
\draw[ thick] (7.2,-1.2) to[out=0,in=180] (7.75,-1.15);
\draw[ thick] (7.75,-1.15) to[out=90,in=-30] (5.39,3.06);
\end{tikzpicture}
\caption{The musquash $\T(c_1)$ with four edges of $\T(c_2)$.}
\label{figure:23}
\end{figure}

Finally, we show that the vertices $v_2$ and $v_3$ cannot be joined by an edge in such a way that the resulting drawing is a thrackle. The edge $v_2v_3$ cannot cross the edges $v_1v_2$, $v_3v_4$, and also, since $c_2$ must be thrackled as a standard musquash, $v_2v_3$ is constrained to cross the edges $vv_1$, $vv_4$ at particular segments. For the drawing on the left in Figure~\ref{figure:23}, the set of arcs that the edge $v_2v_3$ cannot cross is shown in bold in Figure~\ref{figure:24}.

\begin{figure}[h]
\begin{tikzpicture}[scale=0.85]
\draw (4.46,-1.04)-- (6,3.7);
\draw (6,3.7)-- (7.54,-1.04);
\draw (7.54,-1.04)-- (3.51,1.89);
\draw (3.51,1.89)-- (8.49,1.89);
\draw (8.49,1.89)-- (4.46,-1.04);
\draw [thick](6,3.7)-- (5.88,2.76);
\draw [very thick](5.88,2.76)-- (5.22,-0.8);
\fill (4.46,-1.04) circle (1.5pt);
\fill (7.54,-1.04) circle (1.5pt);
\fill (8.49,1.89) circle (1.5pt);
\fill (6,3.7) circle (1.5pt);
\draw (6.26,3.65) node {$v$};
\fill (3.51,1.89) circle (1.5pt);
\fill  (5.22,-0.8) circle (1.5pt);
\draw (5.62,-0.85) node {$v_4$};
\fill  (6.98,-0.4) circle (1.5pt);
\draw (6.98,-0.2) node {$v_1$};
\fill (7.6,1.75) circle (1.5pt);
\draw (7.9,1.7) node {$v_3$};
\draw (7.9,1.15) node {$C$};
\draw (4.4,-1.3) node {$2$};
\draw (7.8,2.1) node {$A$};
\draw (7.35,2.1) node {$B$};
\draw (7.6,-0.7) node {$5$};
\fill  (5.39,3.06) circle (1.5pt);
\draw (5.54,3.32) node {$v_2$};
\draw (3.6,2.14) node {$4$};
\draw (8.66,2.14) node {$3$};
\draw[very thick] (5.22,-0.8) to[out=60,in=90] (8.3,-0.8);
\draw[very thick] (8.3,-0.8) to[out=-90,in=0] (5.4,-1);
\draw[very thick] (5.4,-1) to[out=180,in=-90] (3.4,2.14);
\draw[very thick] (3.4,2.14) to[out=90,in=180] (6.16,3.9);
\draw[very thick] (6.16,3.9) to[out=0,in=90] (7.6,1.75);
\draw[thick] (6,3.7) to[out=-60,in=90] (7.96,-0.2);
\draw[ultra thick] (7.96,-0.2) to[out=-90,in=90] (7.96,-1.1);
\draw[ ultra thick] (7.96,-1.1) to[out=-90,in=0] (7,-1.1);
\draw[ ultra thick] (7,-1.1) to[out=180,in=-90] (6.98,-0.4);
\draw[ ultra thick] (6.98,-0.4) to[out=-85,in=90] (7,-1);
\draw[ ultra thick] (7,-1) to[out=0,in=180] (7.75,-1.15);
\draw[ ultra thick] (7.75,-1.15) to[out=90,in=-30] (5.39,3.06);
\end{tikzpicture}
\caption{The vertices $v_2$ and $v_3$ cannot be joined by an edge.}
\label{figure:24}
\end{figure}
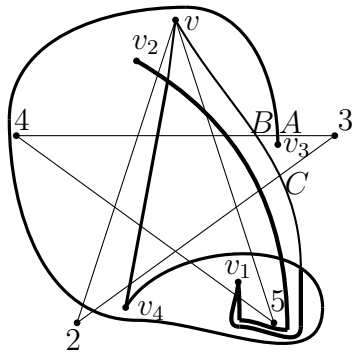

Now the first crossing on the edge $v_3v_2$ starting from the vertex $v_3$ cannot be the crossing with either of the segments $A3$, $C3$, as otherwise the vertex $v_2$ becomes unreachable. So we first cross either $BA$ or $CB$, where in the latter case the next crossing must be with the edge $34$, to the left of $B$. Both of these cases lead to a contradiction: eventually crossing the edge $45$ we will not be able to end up at $v_2$.

Similar arguments for the remaining two drawings in Figure~\ref{figure:23} show that the edge $v_2v_3$ cannot be inserted without violating the thrackle property.
\end{proof}


\section{Concluding remarks}
\label{section:remarks}

In this paper we only considered standard musquashes, the odd ones. By \cite{GD1999} there is the only one even musquash, the thrackled six-cycle. A direct generalisation of the Theorem to the case when $c$ is a six-cycle is false, because $c$ and $c'$ can be disjoint in $G$. The disjoint union of a six-cycle and a three- or a five-cycle \emph{can} be thrackled (which does not violate the Thrackle Conjecture) following the approach in \cite[Section~2]{WOO71}: for both a three- and a five-thrackle, there is a curve which crosses every edge exactly once. We can take a thin strip around that curve and then place the six-musquash inside that strip so that  three vertices are close to the one end, and the others three close to the other end, as in \cite[Figure~6]{WOO71}. However, the figure-eight graph comprised of a three- and a six-cycle cannot be thrackled (by duplicating the three-cycle we get the theta-graph $\Theta_3$, with three paths of length $3$ sharing common endpoints, which has no thrackle drawing by \cite[Theorem~5.1]{LPS97}). However, to the best of our knowledge, the question of whether the figure-eight graph comprised of a five- and a six-cycle can be thrackled remains open.

Finally, we give a more precise statement of Problem~1 in \cite{GY2012}, to which we know no counterexamples: \emph{is it true that the sum of orientations of crossings on any edge of an odd thrackled cycle is $0$, and on any edge of an even thrackled cycle is $\pm 1$?}


\vskip.5cm
\emph{Acknowledgements.} This research would have never been done and this paper would have never been written without invaluable, incredibly generous contribution from Grant Cairns, at all the stages, from mathematics to presentation. We express to him our deepest gratitude.


\bibliographystyle{amsplain}
\bibliography{paper14}

\end{document}